\documentclass[10pt]{amsart}
\usepackage{latexsym, amsmath, amssymb}
\usepackage{version}

\newcommand{\cal}{\mathcal}

\newcommand{\newsection}[1]
{\section{#1}\setcounter{theorem}{0} \setcounter{equation}{0}
\par\noindent}

\newtheorem{theorem}{Theorem}

\newtheorem{lemma}[theorem]{Lemma}
\newtheorem{corr}[theorem]{Corollary}
\newtheorem{prop}[theorem]{Proposition}
\newtheorem{proposition}[theorem]{Proposition}
\newtheorem{deff}[theorem]{Definition}
\newtheorem{remark}[theorem]{Remark}

\newcommand{\C}{{\mathbb C}}
\newcommand{\N}{{\mathbb N}}

\newcommand{\supp}{\text{supp }}
\newcommand{\R}{{\mathbb R}}
\newcommand{\Z}{{\mathbb Z}}

\newcommand{\la}{\langle}
\newcommand{\ra}{\rangle}

\renewcommand{\div}{{\text{div}}}

 \newcommand{\e}{\epsilon}
\newcommand{\dX}{{X}}
\newcommand{\tX}{{{\tilde{X}}}}
\newcommand{\tQ}{{{\tilde{Q}}}}

\renewcommand{\epsilon}{\varepsilon}

\newenvironment{com}{\begin{quotation}{\leftmargin .25in\rightmargin .25in}\sffamily \footnotesize $\clubsuit$}
               {$\spadesuit$\end{quotation}\par\bigskip}

\begin{document}

\title[Strichartz and smoothing estimates for rough Schr\"odinger equations]
{Strichartz estimates and local smoothing estimates for
asymptotically flat Schr\"odinger equations
}

\thanks{The second author was supported by the NSF through a MSPRF,
  and the other two authors were partially supported by the NSF
  through grants DMS0354539 and DMS 0301122}

%\thanks{The authors were supported in part by the NSF}

\author{Jeremy Marzuola}
\author{Jason Metcalfe}
\author{Daniel Tataru}
\address{Department of Mathematics, University of California, Berkeley, CA  94720-3840}

\begin{abstract}
In this article we study global-in-time Strichartz estimates for the
Schr\"o\-dinger evolution corresponding to long-range perturbations
of the Euclidean Laplacian.  This is a natural continuation of
 a recent article \cite{T} of the third author, where it is proved 
that local smoothing estimates imply Strichartz estimates.

By \cite{T} the local smoothing estimates are known to hold for small
perturbations of the Laplacian. Here we consider the case of large
perturbations in three increasingly favorable scenarios: (i) without
non-trapping assumptions we prove estimates outside a compact set
modulo a lower order spatially localized error term, (ii) with
non-trapping assumptions we prove global estimates modulo a lower order
spatially localized error term, and (iii) for time independent
operators with no resonance or eigenvalue at the bottom of the
spectrum we prove global estimates for the projection onto the
continuous spectrum.

% In particular, we make no hypothesis
%   concerning trapping, and we show that lossless Strichartz estimates
%   with a localized error term hold outside a sufficiently large
%   ball.  Moreover, for finite times, this error term can be controlled
%   using energy inequalities, and our estimate provides a $C^2$
%   generalization of the recent local-in-time estimates of Bouclet and
%   Tzvetkov.  Our approach uses a long-time, outgoing parametrix
%   construction of Tataru to reduce the proof of Strichartz estimates
%   to understanding appropriate local smoothing estimates, which is the
%   main focus of the article.
\end{abstract}

%\includeversion{comment}
\includeversion{Q}

\includeversion{QQ}
\newcommand{\bc}{\begin{QQ}\begin{com}}
\newcommand{\ec}{\end{com}\end{QQ}}

\maketitle

%%%%%%%%%%%%%%%%%%%%%%%%%%%%%%%%%%%%%%%%%%%%%%%%%%%%%%%%%%%%%%%%%%%%%%%%%%%%%%%%%%%%%%%%%%%%%%%%%%%%
%%%%%%%%%%%%%%%%%%%%%%%%%%%%%%%%%%%%%%%%%%%%%%%%%%%%%%%%%%%%%%%%%%%%%%%%%%%%%%%%%%%%%%%%%%%%%%%%%%%%
%%%%%%%%%%%%%%%%%%%%%%%%%%%%%%%%%%%%%%%%%%%%%%%%%%%%%%%%%%%%%%%%%%%%%%%%%%%%%%%%%%%%%%%%%%%%%%%%%%%%
\newsection{Introduction}

%\begin{comment}$\clubsuit$\textsf{\footnotesize Text that appears in this font and is surrounded by the club and spade
%symbols represents a detailed note to myself.  These are items that will be removed prior to the final version,
%though they are helpful to me during the process of writing.  The Remarks section at the end contains
%these comments and will be removed for the final version.
%}$\spadesuit$
%\end{comment}

%\begin{Q}
%  \begin{com}
%    Text that appears in this font and is surrounded by the club and spade symbols represents notes which
%I do not intend to appear in the final version.  These can be omitted by recompiling the tex file with
%includeversion replaced by excludeversion just before the maketitle.
%  \end{com}
%\end{Q}

This article is a natural continuation of the third author's work in
\cite{T}, which studies the connection between long-time Strichartz
estimates and local smoothing estimates for Schr\"odinger equations
with $C^2$, asymptotically flat coefficients.

Given a time dependent second order elliptic operator in $\R^n$
\[
A(t,x,D)=D_i a^{ij}(t,x)D_j + b^i(t,x) D_i + D_i b^i(t,x) + c(t,x)
\]
we consider the dispersive properties of solutions to the
Schr\"odinger evolution
\begin{equation}\label{main.equation}
Pu:=(D_t+A(t,x,D))u = f,\quad u(0)=u_0.
\end{equation}

Two of the most stable ways of measuring dispersion are the local
smoothing estimates and the Strichartz estimates.  The local smoothing
estimates give $L^2$ time integrability for the spatially localized
energy, with a half-derivative gain. To state them  we use a local
smoothing space $X$ which will be defined shortly, and its dual $X'$,
\begin{equation}
\| u\|_{X \cap L^\infty_t L^2_x} \lesssim \|u_0\|_{L^2} + \|f\|_{X' + L^1_tL^2_x}
\label{LSE}\end{equation}
where in a first approximation one may set
\[
\| u\|_{X} \sim \| \langle x \rangle^{-\frac12-} |D|^{\frac12} u\|_{L^2_{t,x}}.
\]

The Strichartz estimates on the other hand measure the space-time 
integrability of solutions and have the form
\begin{equation}
\| u\|_{L^{p_1}_t L^{q_1}_x} \lesssim \|u_0\|_{L^2} + \|f\|_{ L^{p'_2}_t L^{q'_2}_x}  
\label{SE} \end{equation}
where the indices $(p_1,q_1)$ and $(p_2,q_2)$ satisfy the relation
\[
\frac{2}{p}+\frac{n}{q}= \frac{n}{2},\qquad 2\le p,q\le \infty
\]
and $(p,q)\neq (2,\infty)$ if $n=2$.  Any pair $(p,q)$ satisfying
these requirements will be called a Strichartz pair.\footnote{For simplicity of exposition, we shall
not directly address the $q=\infty$ endpoint estimate.  This permits us in the sequel to use Littlewood-Paley
theory.  See \cite{KeelTao} for the corresponding endpoint argument in the flat case.}

The local smoothing estimates have been long known to hold in the flat
case $A = -\Delta$ and for certain small perturbations. 
For operators with variable coefficients, local in
time smoothing estimates were first established in \cite{CKS} and
\cite{Doi}.  Global in time estimates on the other hand are
considerably more difficult to obtain and are known only in some very
special cases.  See, e.g., \cite{RT} for time independent, non-trapping, smooth,
compactly supported, though not necessarily small, perturbations of the Laplacian.
 
There are also some known results which show global-in-time smoothing
estimates in the presence of certain trapped rays.  Here, the
estimates involve a different spatial weight and a loss of regularity
due to the trapping.  
See \cite{Hans}, \cite{SSS} and the references
therein.

The Strichartz estimates hold globally in the flat case $A = -\Delta$.
Local-in-time Stri\-chartz estimates for variable coefficient operators
have also been established in \cite{ST}, \cite{HTW}, and \cite{RZ}
provided, amongst other things, that the coefficients are non-trapping.
We also refer the interested reader to the simplified approaches of
\cite{KT} and \cite{Tnotes}. Again, global in time estimates are
more difficult and have been obtained only recently in
\cite{RT} (time independent, non-trapping, smooth, compactly supported perturbations of the
Laplacian) respectively 
\cite{T} (small, $C^2$ long range perturbations of the
Laplacian). 

The above references would be incomplete without mentioning the
vast body of work on dispersive and Strichartz estimates for
lower order perturbations of the Laplacian. For this we refer the reader 
to some of the more recent papers \cite{EGS, EGS2} and the references 
therein.

The third author's article \cite{T} is one of the starting points of
this work. The main result in \cite{T} is to construct a global in
time outgoing parametrix for the equation \eqref{main.equation}  for
$C^2$ long range perturbations of the Laplacian.  This construction
uses the FBI transform, an approach that is reminiscent of the earlier
works \cite{T.I, T.II, T.FP} for the wave equation.  See, also,
\cite{Tnotes} for a survey of these techniques and the closely related
work \cite{Smith} which is based instead on a wave packet
decomposition.

The errors associated to the parametrix are handled using 
the local smoothing estimates. Consequently one is led to the 
second result of \cite{T}, which roughly asserts that
\[
\text{ Local Smoothing Estimates} \implies \text{Strichartz Estimates}.
\]

Local smoothing estimates are also proved in \cite{T}, but only for
small long range perturbations of the Laplacian.  The aim of the
present work is to consider large long range perturbations of the
Laplacian. 

A difficulty one encounters is the possible presence of
trapped rays, i.e. geodesics which are confined to a compact spatial
region. This brings us to our second starting point, namely Bouclet and
Tzvetkov's work \cite{BT}. For smooth, time independent, long range
perturbations of the Laplacian, they prove that local in time
Strichartz estimates hold in the exterior of a sufficiently large
ball, in other words that the loss due to trapping is also confined to
a bounded region.  Another aim of the present work is to provide an analogous
result which is global in time and holds for $C^2$ time-dependent coefficients.

%%%%%%%%%%%%%%%%%%%%%%%%%%%%%%%%%%%%%%%%%%%%%%%%%%
\subsection{Estimates outside a ball}
We begin with our assumptions on the coefficients. Consider a dyadic
spatial decomposition of $\R^n$ into the sets
\[
D_0=\{|x|\le 2\},\quad D_j=\{2^j\le |x|\le 2^{j+1}\}, \quad
j=1,2,\dots
\]
and for $j \geq 0$ set
\[
A_j = \R\times D_j,\quad j\ge 0, \qquad A_{<j}=\R\times \{|x|\le 2^j\}
= \bigcup_{l<j} A_l.
\]
Our weak asymptotic flatness condition has the form
\begin{equation}
\label{coeff}
\sum_{j\in \N}\sup_{A_j} 
\Bigl[\la x\ra^2\Bigl(|\partial_x^2 a(t,x)|+|\partial_t a(t,x)|\Bigr)
+ \la x\ra|\partial_x a(t,x)|+|a(t,x)-I_n|\Bigr]\le \kappa <\infty
\end{equation}
and for the lower order terms we have a related condition,
\begin{equation}\label{coeffb}
\sum_{j\in \N} \sup_{A_j} \la x\ra
|b(t,x)|\le \kappa
\end{equation}
\begin{equation}
\label{coeffc}
\begin{split}
&\left\{\begin{array}{l}
\sup \la x\ra^2 (|c(t,x)|+|\div\  b(t,x)|) \le \kappa
\cr\cr   \displaystyle \limsup_{|x| \to \infty}\
 \la x\ra^2 (|c(t,x)|+|\div\ b(t,x)|) <  \epsilon \ll 1 \end{array} 
\right. \qquad \qquad n \neq 2
\\ 
&\left\{\begin{array}{l}\sup \la x\ra^2 (\ln(2+|x|^2))^2
    (|c(t,x)|+|\div\  b(t,x)|) \le \kappa,  
\cr \cr 
\displaystyle \limsup_{|x| \to \infty}\
 \la x\ra^2 (\ln \la x\ra )^2 (|c(t,x)|+|\div\  b(t,x)|) <  \epsilon \ll 1
\end{array}\right.
 \qquad n=2.
\end{split}
\end{equation}
% We note that we could also replace \eqref{coeffc} by
% \begin{equation}\label{coeffc'}
%   \sup_{\R\times\R^n} \la x\ra^2 |c(t,x)|\le \varepsilon,
% \end{equation}
% but we would need to assume that $\varepsilon$ is sufficiently small.

Here $\epsilon$ is a fixed sufficiently small parameter.  For any
$\kappa$, \eqref{coeff} restricts the trapped rays to finitely many of
the regions $A_j$.  If $\kappa$ is sufficiently small, which we do not
assume, then it is known that trapped rays do not exist.  Notice that
we may choose $M=M(\varepsilon)$ sufficiently large so that
\begin{equation}\label{coeff.M}
\sum_{j\ge M} \sup_{A_j}\Bigl[\la x\ra^2 \Bigl(|\partial_x^2 a(t,x)|+|\partial_t a(t,x)|\Bigr)
+\la x\ra|\partial_x a(t,x)|+|a(t,x)-I_n|\Bigr]\le \varepsilon\end{equation}
and
\begin{equation}\label{coeffb.M}
\sum_{j \geq M} \sup_{A_j} \la x\ra
|b(t,x)|\le \epsilon 
\end{equation}
\begin{equation}
\label{coeffc.M}
\begin{split}
  & \sup_{A_{\geq M}} \la x\ra^2 (|c(t,x)|+|\div\ b(t,x)|) \le
  \epsilon, \qquad n \neq 2
  \\
  &\sup_{A_{\geq M}} \la x\ra^2 (\ln \la x\ra)^2 (|c(t,x)|+|\div\
  b(t,x)|) \le \epsilon, \qquad n = 2.
\end{split}
\end{equation}

To describe the local smoothing space $\dX$, we  use a dyadic partition of unity of frequency
\[
1=\sum_{k=-\infty}^\infty S_k(D).
\]
The functions at frequency $2^k$ are measured using the norms
\[
\|u\|_{X_k}= \|u\|_{L^2_{t,x}(A_{<0})} + \sup_{j\ge 0} \|\la x\ra^{-1/2} u\|_{L^2_{t,x}(A_j)},
\quad k>0
\]
\[
\|u\|_{X_k}=2^{\frac{k}2} \|u\|_{L^2_{t,x}(A_{<-k})} +  \sup_{j\ge -k} \|(|x|+2^{-k})^{-1/2} u\|_{L^2_{t,x}(A_j)},
\quad k\le 0.
\]
The local smoothing space $\dX$ is the completion of the Schwartz space
with respect to the norm
\[
\|u\|_{\dX}^2 = \sum_{k=-\infty}^\infty 2^k \|S_k u\|^2_{X_k}.
\]
Its dual $\dX'$ has norm 
\[
\|f\|_{\dX'}^2=\sum_{k=-\infty}^\infty 2^{-k} \|S_k f\|^2_{X_k'}.
\]

In dimension $n \geq 3$ the space $\dX$ is a space of distributions,
and we have the Hardy type inequality
\begin{equation}
\| \la x\ra^{-1} u\|_{L^2_{t,x}} \lesssim \| u\|_{\dX}.
\label{Hardy}\end{equation}
On the other hand in dimensions $n=1,2$, the space $\dX$ is a space of
distributions modulo constants, and we have the BMO type inequality
\begin{equation}
\sum_{j\ge 0}\| \la x\ra^{-1} (u-u_{D_j})\|^2_{L^2_{t,x}(A_j)} \lesssim \| u\|^2_{\dX}
\label{Hardylow}\end{equation}
where $u_{D_j}$ represents the (time dependent) average 
of $u$ in $D_j$.
 At the same time $\dX'$ contains only
functions with integral zero.  We refer the reader to \cite{T} for
more details.

In \cite{T} the case of a small perturbation
of the Laplacian is considered, and it is proved that
\begin{theorem}~\cite{T}. Assume that either

(i) $n \geq 3$ and \eqref{coeff}, \eqref{coeffb},\eqref{coeffc} hold
with a sufficiently small $\kappa$ or 

(ii) $n = 1,2$, $b^i=0$, $c=0$ and   \eqref{coeff} holds
with a sufficiently small $\kappa$.

Then the  local smoothing estimate 
\begin{equation}
  \label{lsesmall}
\| u\|_{\dX \cap L^\infty_t L^2_x} \lesssim \|u_0\|_{L^2}
+ \|f\|_{\dX'+L^1_tL^2_x} 
\end{equation}
holds for all solutions $u$ to \eqref{main.equation}.
\end{theorem}

As one can see, the assumptions are more restrictive in low
dimensions.  This is related to the spectral structure of the operator
$A$, precisely to the presence of a resonance at zero. This is the case
if $A = -\Delta$ or, more generally, if $b^i=0$ and $c=0$. However
the zero resonance is unstable with respect to lower order perturbations.
To account for non-resonant situations,  it is convenient to 
introduce a stronger norm which removes the quotient structure,
\[
\begin{split}
\|u\|_{\tX}^2 = &\  \| \la x\ra^{-1} u\|_{L^2_{t,x}}^2 + 
 \sum_{k=-\infty}^\infty 2^k \|S_k u\|^2_{X_k}, \qquad n \neq 2
\\
\|u\|_{\tX}^2 = &\  \| \la x\ra^{-1} (\ln (2+|x|))^{-1} u\|_{L^2_{t,x}}^2 + 
 \sum_{k=-\infty}^\infty 2^k \|S_k u\|^2_{X_k}, \qquad n = 2.
\end{split}
\]
Its dual is 
\[
\tX' = \dX' + \la x \ra L^2_{t,x}, \quad n \neq 2,\qquad  \tX' = \dX' + \la x
  \ra (\ln(2+|x|)) L^2_{t,x}, \quad n = 2.
\]
Due to the Hardy inequality above, if $n \geq 3$ we have $\tX = \dX$. 
On the other hand in low dimension the $\tX$
 norm adds some local
square integrability to the $\dX$ norm. Precisely, we have 
\begin{lemma}
Let $n=1,2$. Then
\begin{equation}
\| u\|_{\tX} \lesssim \|u\|_{\dX} + \|u\|_{L^2_{t,x}(\{|x| \leq 1\})}.
\end{equation}
\label{txdx}
\end{lemma}

The first goal of this article is to show, without any trapping
assumption, that loss-less (with respect to regularity), global-in-time local smoothing and
Strichartz estimates hold exterior to a sufficiently large ball,
modulo a localized error term.  It is hoped that this error term can
be separately estimated for applications of interest.  Moreover, in
the case of finite times, this error term can be trivially estimated
by the energy inequality and immediately yields a $C^2$, long range,
time dependent analog of the result of \cite{BT}.

For $M$ fixed and sufficiently large so that %amongst other things
\eqref{coeff.M}, \eqref{coeffb.M} and \eqref{coeffc.M} hold,
we consider a smooth, radial, nondecreasing cutoff function $\rho$ which is supported in
$\{|x|\ge 2^M \}$ with $\rho(|x|)\equiv 1$ for $|x|\ge 2^{M+1}$.  
%\begin{QQ}
%  \begin{com}
%    Need $\rho$ to be increasing.
%  \end{com}
%\end{QQ}
Then we define the exterior local smoothing space $\tX_e$
with norm
\[
\| u\|_{\tX_e} = \| \rho u\|_{\tX} + \| (1-\rho) u\|_{L^2_{t,x}}
\]
and the dual space $\tX'_e$ with norm 
\[
\|f\|_{\tX'_e} = \inf_{f = \rho f_1 +
  (1-\rho) f_2} \| f_1\|_{\tX'} + \|f_2\|_{L^2_{t,x}}. 
\]
Now we can state our exterior local smoothing estimates.
\begin{theorem}
  \label{main.ls.theorem}
  Let $n \geq 1$.  Assume that the coefficients $a^{ij}$, $b^i$ and
  $c$ are real and satisfy \eqref{coeff}, \eqref{coeffb},
  \eqref{coeffc}. Then the solution $u$ to \eqref{main.equation}
  satisfies
\begin{equation}
  \label{lsext}
\| u\|_{\tX_e \cap L^\infty_t L^2_x} \lesssim \|u_0\|_{L^2}
+ \|f\|_{\tX'_e+L^1_tL^2_x} + 
 \|u\|_{L^2_{t,x}(\{|x|\le 2^{M+1}\})}.
\end{equation}
\end{theorem}

In the low dimensional resonant case the situation is a bit more
delicate.  First of all, the above theorem does not give a meaningful estimate in the 
$n=1,2$ resonant case as the last term in the right of \eqref{lsext} blows up for
constant functions, which correspond to the zero resonance.
 Since we do not control the local $L^2$ norm for $\dX$
functions, truncation by the cutoff function $\rho$ does not preserve
the $\dX$ space. To remedy this we define a time dependent local
average for $u$, namely
\[
u_\rho =   \left( \int_{\R^n} (1-\rho) \ dx \right)^{-1}    \int_{\R^n} (1-\rho) u\ dx,
\]
and define a modified truncation by the self-adjoint operator
\[
T_\rho u = \rho u + (1-\rho) u_\rho.
\]
We note that $T_\rho$ leaves constant functions unchanged, as well as the
integral of $u$ (if finite).

Then we set
\[
\| u\|_{\dX_e} = \| T_\rho u\|_{\dX} + \|  u - T_\rho u \|_{L^2_{t,x}}
\]
and have the dual space $\dX'_e$ with norm 
\[
\|f\|_{\dX'_e} = \inf_{f = T_\rho f_1  +  (1-T_\rho) f_2} 
\| f_1\|_{\dX'} + \|f_2\|_{L^2_{t,x}}. 
\]
We now have the following alternative to
Theorem~\ref{main.ls.theorem} which is consistent with operators with
a constant zero resonance:

\begin{theorem}
  \label{main.ls.theorem.res}
  Let $n = 1,2$.  Assume that 

(i) the coefficients $a^{ij}$ are real and satisfy \eqref{coeff};

(ii) the coefficients $b^i$ are real, satisfy \eqref{coeffb},
and $\partial_i b^i = 0$;

(iii) there are no zero order terms, $c = 0$.

 Then the solution $u$ to \eqref{main.equation}
 satisfies
\begin{equation}
  \label{lsext.res}
\| u\|_{\dX_e \cap L^\infty_t L^2_x} \lesssim \|u_0\|_{L^2}
+ \|f\|_{\dX'_e+L^1_tL^2_x} + 
 \|u - u_\rho\|_{L^2_{t,x}(\{|x|\le 2^{M+1}\})}.
\end{equation}
\end{theorem}

Once we have the local smoothing estimates, the parametrix
construction in \cite{T} allows us to obtain corresponding Strichartz
estimates.  If $(p,q)$ is a Strichartz pair we define the exterior
space $\tX_e(p,q)$ with norm
\[
\| u\|_{\tX_e(p,q)} = \| u\|_{\tX_e} + \| \rho u\|_{L^p_t L^q_x}
\]
and the dual space $\tX'(p,q)$ with norm 
\[
\|f\|_{\tX'_e(p,q)} = \inf_{f =  f_1 +  \rho f_2} \| f_1\|_{\tX'_e} + 
\|f_2\|_{L^{p'}_t L^{q'}_x}.  
\]

\begin{theorem}
  \label{main.est.theorem}
   Let $n \geq 1$.  Assume that the coefficients $a^{ij}$, $b^i$ and
  $c$ are real and satisfy \eqref{coeff}, \eqref{coeffb},
  \eqref{coeffc}.  Then for any two Strichartz pairs
  $(p_1,q_1)$ and $(p_2, q_2)$, the solution $u$ to
  \eqref{main.equation} satisfies
\begin{equation}
  \label{fse}
\| u\|_{\tX_e(p_1,q_1) \cap L^\infty_t L^2_x} \lesssim \|u_0\|_{L^2}
+ \|f\|_{\tX'_e(p_2,q_2)+L^1_tL^2_x} + 
 \|u\|_{L^2_{t,x}(\{|x|\le 2^{M+1}\})}.
\end{equation}
\end{theorem}

Correspondingly, in the resonant case we define
\[
\| u\|_{\dX_e(p,q)} = \| u\|_{\dX_e} + \| \rho u\|_{L^p_t L^q_x}
\]
and the dual space $\dX'_e(p,q)$ with norm 
\[
\|f\|_{\dX'_e(p,q)} = \inf_{f =  f_1 +  \rho f_2} \| f_1\|_{\dX'_e} + 
\|f_2\|_{L^{p'}_t L^{q'}_x}.  
\]
Then we have 

\begin{theorem}
  \label{main.est.theorem.res}
  Let $n = 1,2$. Assume that the coefficients of $P$ are as in
  Theorem~\ref{main.ls.theorem.res}.  Then for any two Strichartz
  pairs $(p_1,q_1)$ and $(p_2, q_2)$, the solution $u$ to
  \eqref{main.equation} satisfies
\begin{equation}
  \label{fselow}\| u\|_{\dX_e(p_1,q_1)\cap L^\infty_t L^2_x} \lesssim 
\|u_0\|_{L^2}  + \|f\|_{\dX'_e(p_2,q_2)+L^1_t L^2_x} 
  +  \|u-u_\rho\|_{L^2_{t,x}(\{|x|\le 2^{M+1}\})}.
\end{equation}
\end{theorem}

In both cases the space-time norms are over $[0,T]\times \R^n$ for any
time $T>0$ with constants independent of $T$.
If the time $T$ is finite, then we may use energy estimates to
trivially bound the error term.  Doing so results in the following,
which is a $C^2$-analog of the exterior Strichartz estimates of
\cite{BT}.
\begin{corr}
  \label{corr.main.theorem}
(a.)  Assume that the coefficients $a^{ij}$, $b^i$, and $c$ are as in
  Theorem~\ref{main.ls.theorem}.  Then for any two Strichartz pairs
  $(p_1,q_1)$ and $(p_2, q_2)$, the solution $u$ to
  \eqref{main.equation} satisfies
\begin{equation}
  \label{fse2}
\| u\|_{\tX(p_1,q_1)\cap L^\infty_t L^2_x} \lesssim_T \|u_0\|_{L^2}
+ \|f\|_{\tX'(p_2,q_2)+L^1_t L^2_x}.
\end{equation}

(b.) Assume that the coefficients $a^{ij}$ and $b^i$ are as in
Theorem~\ref{main.ls.theorem.res}.  Then for any two Strichartz pairs
$(p_1,q_1)$ and $(p_2,q_2)$, the solution $u$ to \eqref{main.equation}
satisfies
\begin{equation}
  \label{fse2.res}
\|u\|_{\dX(p_1,q_1)\cap L^\infty_t L^2_x}\lesssim_T \|u_0\|_{L^2}
+ \|f\|_{\dX'(p_2,q_2)+L^1_tL^2_x}.
\end{equation}

In both cases, the space-time norms are over $[0,T]\times \R^n$ and 
$T>0$ is finite.
\end{corr}

We conclude this subsection with a few remarks concerning 
several alternative set-ups for these results.

%%%%%%%%%%%%%%%%%%%%%%%%%%%%%%%%%%%%%%%%%%%%%%%%%%
\subsubsection{ Boundary value problems }

Our proof of Theorems~\ref{main.ls.theorem},\ref{main.ls.theorem.res},
\ref{main.est.theorem}~\ref{main.est.theorem.res} treats the interior 
of the ball $B = \{ |x| < 2^M \}$ as a black box with the sole property
that the energy is conserved by the evolution.
Hence the results remain valid for exterior boundary problems.
Precisely, take a bounded domain $\Omega \subset B$
and consider either the Dirichlet problem
\begin{equation}
\left\{ \begin{array}{lc}  Pu = f & \text{in } \Omega^c 
\cr u(0) = u_0 & \cr 
u = 0&  \text{in } \partial \Omega \end{array} \right.
\label{D}\end{equation}
or the Neumann problem
\begin{equation}
\left\{ \begin{array}{lc}  Pu = f & \text{in } \Omega^c 
\cr u(0) = u_0 & \cr \displaystyle 
\frac{\partial u}{\partial \nu} = 0&  \text{in } \partial \Omega
 \end{array} \right.
\label{N}\end{equation}
where
\[
\frac{\partial}{\partial \nu} = \nu_i (a^{ij} D_j +  b^i)
\]
and $\nu$ is the unit normal to $ \partial \Omega$.

Then we have 
\begin{corr}
  a) The results in Theorems~\ref{main.ls.theorem} and
  \ref{main.est.theorem} remain valid for both the Dirichlet
  problem \eqref{D} and the Neumann problem \eqref{N}.

  b) The results in Theorems~ \ref{main.ls.theorem.res} and
  \ref{main.est.theorem.res} remain valid for  
the Neumann problem \eqref{N} with the additional condition
$b^i \nu_i = 0$ on $ \partial \Omega$.
\end{corr}

The more restrictive hypothesis in part (b) is caused by the 
requirement that constant functions solve the homogeneous
problem.

%%%%%%%%%%%%%%%%%%%%%%%%%%%%%%%%%%%%%%%%%%%%%%%%%%
\subsubsection{Complex coefficients} 

The only role played in our proofs by the assumption
that the coefficients $b^i$ and $c$ are real  is to insure 
the energy conservation in the interior region. Hence 
we can allow complex coefficients in the region $\{ |x| > 2^{M+1}\}$
where the coefficients satisfy the smallness condition.

In addition, allowing $c$ to be complex in the interior region does
not affect energy conservation either, since we are assuming 
an a priori control of the local $L^2$ space-time norm of the solution.
Hence we have

\begin{remark}
a)  The results in Theorems~\ref{main.ls.theorem} and 
\ref{main.est.theorem} remain valid for complex coefficients
 $b^i$, $c$  with the restriction that $b^i$ are real  
in the region $\{ |x| < 2^{M+1}\}$.

b) The results in Theorems~ \ref{main.ls.theorem.res} and
\ref{main.est.theorem.res} remain valid for coefficients $b^i$ which
are real in the region $\{ |x| < 2^{M+1}\}$.
\end{remark}

%%%%%%%%%%%%%%%%%%%%%%%%%%%%%%%%%%%%%%%%%%%%%%%%%%
\subsection{ Non-trapping metrics}
The second goal of the article is to consider the previous setup but
with an additional non-trapping assumption. To state it we consider
the Hamilton flow $H_a$ for the principal symbol of the operator 
$A$, namely 
\[
a(t,x,\xi) = a^{ij}(t,x) \xi_i \xi_j.
\]
The spatial projections of the trajectories of the Hamilton flow $H_a$
are the geodesics for the metric $a_{ij} d x^i dx^j$ where
$(a_{ij})=(a^{ij})^{-1}$.

\begin{deff}
  We say that the metric $(a_{ij})$ is non-trapping if for each $R > 0$
  there exists $L > 0$ independent of $t$ so that any portion of a
  geodesic contained in $\{|x| < R\}$ has length at most $L$.
\end{deff}

The non-trapping condition allows us to use standard propagation of
singularities techniques to bound high frequencies inside a ball in
terms of the high frequencies outside.  Then the cutoff function
$\rho$ which was used before is no longer needed, and we obtain
\begin{theorem}
  \label{main.ls.theoremnt}
  Let $R > 0$ be sufficiently large.  Assume that the coefficients $a^{ij}$, $b^i$ and $c$ are real and
  satisfy \eqref{coeff}, \eqref{coeffb}, \eqref{coeffc}.  Assume also
  that the metric $a_{ij}$ is non-trapping.  Then the solution $u$ to
  \eqref{main.equation} satisfies
\begin{equation}
  \label{lsnt}
\| u\|_{\tX} \lesssim \|u_0\|_{L^2}
+ \|f\|_{\tX'} + 
 \|u\|_{L^2_{t,x}(\{|x|\le 2R\})},
\end{equation}
\end{theorem}

respectively 

\begin{theorem}
  \label{main.ls.theoremnt.res}
  Let $R > 0$ be sufficiently large, and let $n=1,2$.  Assume that the coefficients of $P$ are as
  in Theorem~\ref{main.ls.theorem.res}.  Assume also that the metric
  $a_{ij}$ is non-trapping.  Then the solution $u$ to
  \eqref{main.equation} satisfies
\begin{equation}
  \label{lsntres}
\| u\|_{\dX} \lesssim \|u_0\|_{L^2}
+ \|f\|_{\dX'} + 
 \|u-u_\rho\|_{L^2_{t,x}(\{|x|\le 2R\})}.
\end{equation}
\end{theorem}

We note that the high frequencies in the error term on the right are
controlled by the $\dX$ norm on the left. Also the low frequencies ($\ll
1$) are controlled by the $\dX$ norm using the uncertainty
principle. Hence the only nontrivial part of the error term corresponds
to intermediate (i.e. $\approx 1$ ) frequencies.

The proof combines the arguments used for the exterior estimates
with a standard multiplier construction from the theory of propagation
of singularities. Adding to  the above results  the parametrix
obtained in \cite{T} we obtain

\begin{theorem}
  \label{main.theoremnt}
  Let $R > 0$ be sufficiently large.  Assume that the coefficients $a^{ij}$, $b^i$ and
  $c$ are real and satisfy \eqref{coeff}, \eqref{coeffb}, \eqref{coeffc}.  Assume
  also that the metric $a_{ij}$ is non-trapping.  Then for any two
  Strichartz pairs $(p_1,q_1)$ and $(p_2, q_2)$, the solution $u$ to
  \eqref{main.equation} satisfies
\begin{equation}
  \label{fsent}
\| u\|_{\tX \cap L^{p_1}_t L^{q_1}_x} \lesssim \|u_0\|_{L^2}
+ \|f\|_{\tX'+L^{p'_2}_t L^{q'_2}_x} + 
 \|u\|_{L^2_{t,x}(\{|x|\le 2R\})},
\end{equation}
\end{theorem}

respectively 

\begin{theorem}
  \label{main.theoremnt.res}
 Let $n = 1,2$, and let $R > 0$ be sufficiently large.  Assume that the coefficients of $P$ are as
  in Theorem~\ref{main.ls.theorem.res}. Assume also that the metric
$a_{ij}$ is non-trapping. Then for any two Strichartz pairs $(p_1,q_1)$
and $(p_2, q_2)$, the solution $u$ to \eqref{main.equation} satisfies
\begin{equation}
  \label{fsentres}\| u\|_{\dX\cap L^{p_1}_t L^{q_1}_x} \lesssim \|u_0\|_{L^2}
  + \|f\|_{\dX'+L^{p'_2}_t L^{q'_2}_x} 
  +  \| u - u_\rho\|_{L^2_{t,x}(\{|x|\le 2R\})}.
\end{equation}
\end{theorem}

%%%%%%%%%%%%%%%%%%%%%%%%%%%%%%%%%%%%%%%%%%%%%%%%%%
\subsubsection{ An improved result for trapped metrics}

A variation on the above theme is obtained in the case when there are
trapped rays, but not too many. If they exist, they must be confined
to the interior region $\{|x| \leq 2^M\}$. Then we can define the
conic set
\[
\begin{split}
  \Omega_{trapped}^L = \{&\ (t,x,\xi) \in \R \times T^* B(0,2^M);
  \text{ the $H_a$ bicharacteristic through } (t,x,\xi)
\\ & \text{ has
    length at least } L \text{ within } |x| \leq 2^M\}.
\end{split}
\]
Given a smooth zero homogeneous symbol $q(x,\xi)$  which 
equals $1$ for $|x| > 2^M$, we define modified exterior spaces
by
\[
\| u\|_{\tX_q}  = \| q(x,D) u\|_{\tX} + \| u\|_{L^2(\{|x| \leq 2^{M+1}\})} 
\]
with similar modifications for $\tX'_q$, $\dX_q$ and $\dX'_q$.

Then the same argument as in the proof of the above Theorems
gives

\begin{corr} 
  Assume that $q$ is supported outside $ \Omega_{trapped}^L$ for some
  $L > 0$. Then the  results in Theorems~\ref{main.ls.theorem},
  \ref{main.est.theorem}, \ref{main.ls.theorem.res} and
  \ref{main.est.theorem.res} remain valid with $\tX_e$, $\tX'_e$,
  $\dX_e$ and $\dX'_e$ replaced by $\tX_q$, $\tX'_q$, $\dX_q$ and
  $\dX'_q$.
\end{corr}

We also note that if $A$ has time independent coefficients then
$\Omega_{trapped}^L$ is translation invariant. Hence a compactness
argument allows us to replace $\Omega_{trapped}^L$ by
$\Omega_{trapped}^\infty$, which contains all the trapped geodesics.

%%%%%%%%%%%%%%%%%%%%%%%%%%%%%%%%%%%%%%%%%%%%%%%%%%
\subsubsection{Boundary value problems}

Consider solutions $u$ for either the Dirichlet problem \eqref{D} or
the Neumann problem \eqref{N}. Then singularities will propagate along
generalized broken bicharacteristics (see \cite{MS1, MS2},\cite{Hor},\cite{BGT2}).
Hence the non-trapping condition needs to be modified accordingly.

\begin{deff}
  We say that the metric $(a_{ij})$ is non-trapping if for each $R > 0$
  there exists $L > 0$ independent of $t$ so that any portion of a
  generalized broken bicharacteristic is contained in $\{|x| < R\}$
  has length at most $L$.
\end{deff}

With this modification the results of
Theorems~\ref{main.ls.theoremnt}, \ref{main.ls.theoremnt.res}, remain
valid.  However, some care must be taken with the results on
propagation of singularities near the boundary, as not all of them are
known to be valid for operators with only $C^2$ coefficients.

On the other hand we do not know whether the bounds in Theorems
\ref{main.theoremnt}, \ref{main.theoremnt.res} are true or not.  These
hinge on the validity of local Strichartz estimates near the boundary.
This is currently an unsolved problem.

%%%%%%%%%%%%%%%%%%%%%%%%%%%%%%%%%%%%%%%%%%%%%%%%%%
\subsubsection{ Complex coefficients}

Again, one may ask to what extent are our results in this section 
are valid if complex coefficients are allowed. We have 
\begin{remark}
  The results in Theorems~\ref{main.ls.theoremnt},
  \ref{main.ls.theoremnt.res}, \ref{main.theoremnt},
  \ref{main.theoremnt.res} remain valid if the coefficients $b^i$ and
  $c$ are allowed to be complex.

\end{remark}

This result is obtained without making any changes to our proofs
provided that the constant $\kappa$ in \eqref{coeffb} is sufficiently
small. Otherwise, the multiplier $q$ used in the proof has to change 
too much along bicharacteristics from entry to exit from 
$B(0,2^M)$; this in turn forces a modified multiplier for the exterior
region.  See, e.g., \cite{D1, Doi} and \cite{ST}.

%%%%%%%%%%%%%%%%%%%%%%%%%%%%%%%%%%%%%%%%%%%%%%%%%%%%%%%%%%%%%%%%%%%%%%%%%%%
\subsection{ Time independent metrics}

It is natural to ask when can one eliminate the error term altogether.
This is a very delicate question, which hinges on the local in space
evolution of low frequency solutions. For general operators $A$ with 
time dependent coefficients this question seems out of reach for now.

This leads us to the third part of the paper where, in addition to the
flatness assumption above and the non-trapping hypothesis on $a_{ij}$, we
take our coefficients $a^{ij},b^i,c$ to be time-independent.  Then the natural 
obstruction to the dispersive estimates comes from possible
eigenvalues and zero resonances of the operator $A$.

Since the operator $A$ is self-adjoint, it follows that its spectrum is
real. More precisely, $A$ has a continuous spectrum $\sigma_c=
[0,\infty)$ and a point spectrum $\sigma_p$ consisting of discrete
finite multiplicity eigenvalues in $\R^-$, whose only possible
accumulation point is $0$.  

From the point of view of dispersion there is nothing we can do about
eigenvalues. Consequently we introduce the spectral projector $P_c$
onto the continuous spectrum, and obtain dispersive estimates only for
$P_c u$ for solutions $u$ to \eqref{main.equation}. 

The resolvent 
\[
R_\lambda = (\lambda - A)^{-1}
\]
is well defined in $\C \setminus (\sigma_c \cup \sigma_p)$.
One may ask whether there is any meromorphic continuation 
of the resolvent $R_\lambda$ across the positive real axis, starting
on either side. This is indeed possible. The poles of this
meromorphic continuation are called resonances.
This is of interest to us because the resonances which are close to
the real axis play an important role in the long time behavior
of solutions to the Schr\"{o}dinger equation. 

In the case which we consider here (asymptotically flat), there are no
resonances nor eigenvalues inside the continuous spectrum i.e. in
$(0,\infty)$. However, the bottom of the continuous spectrum,
namely $0$, may be either an eigenfunction (if $n \geq 5$) or 
a resonance (if $n \leq 4$). For zero resonances we use
a fairly restrictive definition:

\begin{deff}
We say that $0$ is a resonance for $A$ if there is a function 
$u \in \tX^0$ so that $Au = 0$. The function $u$ 
is called a zero resonant state of $A$.
\end{deff}

Here $\tX^0$ denotes the spatial part of the $\tX$ norm.  I.e. $\tX = L^2_t \tX^0$.

The main case we consider here is when $0$ is neither an eigenfunction
(if $n \geq 5$) nor a resonance (if $n \leq 4$). This implies that
there are no eigenvalues close to $0$. Then $A$ has at most finitely
many negative eigenvalues, and the corresponding eigenfunctions decay
exponentially at infinity.

\begin{theorem}\label{theorem.PcSmoothing}  
  Suppose that $a^{ij},b^i,c$ are real, time-independent, and satisfy the
  conditions \eqref{coeff},\eqref{coeffb}, and
  \eqref{coeffc}.  We also assume that the Hamiltonian vector field
  $H_a$ permits no trapped geodesics and that $0$ is not an
  eigenvalue or a resonance of $A$.  Then for all solutions $u$ to
  \eqref{main.equation} we have
\begin{equation}\label{PcSmoothing}
\|P_c u\|_\tX\lesssim \|u_0\|_2 + \| f\|_{\tX'}.
\end{equation}
\end{theorem}
From this, using the parametrix of \cite{T}, we immediately obtain the corresponding global-in-time
Strichartz estimates:
\begin{theorem}\label{corr.nontrap.Strichartz}
  %Let $n\ge 3$, and 
  Suppose that $a^{ij},b^i,c$ are real, time-independent, and
  satisfy the  conditions \eqref{coeff},\eqref{coeffb},
  and \eqref{coeffc}.  Moreover, assume that the Hamiltonian vector
  field $H_a$ permits no trapped geodesics.  Assume, also, that $0$ is
  not an eigenvalue or a resonance of $A$.  Then for all solutions $u$
  to \eqref{main.equation}, we have
\begin{equation}\label{PcStrichartz}
\|P_c u\|_{L^{p_1}_t L^{q_1}_x \cap \tX} \lesssim \|u_0\|_2  + \|f\|_{L^{p'_2}_t
  L^{q'_2}_x+\tX'} ,
\end{equation}
for any Strichartz pairs $(p_1,q_1)$ and $(p_2,q_2)$.
\end{theorem}

One can compare this with the result of \cite{RT}, where the authors
consider a smooth compactly supported perturbation of the metric in
$3+1$ dimensions where no eigenvalues are present.  Estimates in the
spirit of \eqref{PcStrichartz} have also recently be shown by
\cite{BT2}, though only for smooth coefficients and with a more
restrictive spectral projection.  We also note the related work
\cite{EGS} on Schr\"odinger equations with magnetic potentials.  In
their work, the second order operator is taken to be $-\Delta$.
Theorem \ref{corr.nontrap.Strichartz} is a more general version of
the main theorem in \cite{EGS} in the sense that it allows a more
general leading order operator and that it assumes less flatness on
the coefficients.

In dimension $n \geq 3$ zero is not an eigenvalue or a resonance 
for $-\Delta$, nor for small perturbations of it. However, in
dimension $n=1,2$, zero is a resonance and the corresponding
resonant states are the constant functions. This spectral picture
is not stable with respect to lower order perturbations, but it does
remain stable with respect to perturbations of the metric $a^{ij}$.
Hence there is some motivation to also investigate this case 
in more detail. We prove the following result.

\begin{theorem}\label{theorem.PcSmoothing.res}  
  Assume that the coefficients of $P$ are time-independent, but otherwise as in
  Theorem~\ref{main.ls.theorem.res}.  Assume also that the Hamiltonian
  vector field $H_a$ permits no trapped geodesics, and that there are
  no nonconstant zero resonant states of $A$.  Then for all
  solutions $u$ to \eqref{main.equation}, we have
\begin{equation}\label{PcSmoothing.res}
\| u\|_\dX\lesssim \|u_0\|_2 + \| f\|_{\dX'}.
\end{equation}
\end{theorem}

In terms of Strichartz estimates, this has the following consequence:

\begin{theorem}\label{corr.nontrap.Strichartz.res}
  Assume that the coefficients of $P$ are time-independent, but otherwise as in
  Theorem~\ref{main.ls.theorem.res}.  Assume also that the Hamiltonian
  vector field $H_a$ permits no trapped geodesics, and that there are
  no nonconstant zero resonant states of $A$.  Then for all
  solutions $u$ to \eqref{main.equation}, we have
\begin{equation}\label{PcStrichartz.res}
\| u\|_{L^{p_1}_t L^{q_1}_x \cap \dX} \lesssim \|u_0\|_2  + \|f\|_{L^{p'_2}_t
  L^{q'_2}_x+\dX'} 
\end{equation}
for any Strichartz pairs $(p_1,q_1)$ and $(p_2,q_2)$.
\end{theorem}

Implicit in the above theorems is the fact that there are, under their
hypothesis, no eigenvalues for $A$.  There is another
simplification if we make the additional assumption that $b = 0$.

\begin{remark}\label{remark.nootherres}
If in addition $b = 0$, then there are no nonconstant generalized zero
eigenvalues of $A$.
\end{remark}

In order to prove Theorems \ref{theorem.PcSmoothing} and
\ref{theorem.PcSmoothing.res}, we restate the bounds
\eqref{PcSmoothing} and \eqref{PcSmoothing.res} in terms of estimates
on the resolvent using the Fourier transform in $t$.  We then argue
via contradiction.  Using the positive commutator method, we show an
outgoing radiation condition (see Steps 8-10 of the proof), which allows us
to pass to subsequences and claim that if \eqref{PcSmoothing} were
false, then there is a resonance or an eigenvalue $v$ within the
continuous spectrum. By hypothesis this cannot occur at $0$.  We use
another multiplier and the radiation condition to then show that $v\in
L^2$ and thus cannot be a resonance.  As results of \cite{KT2} show
that there are no eigenvalues embedded in the continuous spectrum, we
reach a contradiction. If instead \eqref{PcSmoothing.res} were false,
then the same argument produces a nonconstant zero resonance, again
reaching a contradiction.
 
The paper is organized as follows.  In the next section, we fix some
further notations and our paradifferential setup.  It is here that we
show that we may permit the lower order terms in the local smoothing
estimates in a perturbative manner.
In the third section, we prove the local smoothing estimates using the
positive commutator method, first in the exterior local smoothing
spaces and then in the non-trapping case.  The fourth section is
devoted to non-trapping, time-independent operators.
 In the final section, we review the  parametrix of \cite{T} 
and use it to  show how the Strichartz estimates follow
from the local smoothing estimates. 

{\em Acknowledgements:} The authors thank W. Schlag and M. Zworski for helpful discussions
regarding some of the spectral theory, and in particular the behavior of resonances, contained herein.

\bigskip
%%%%%%%%%%%%%%%%%%%%%%%%%%%%%%%%%%%%%%%%%%%%%%%%%%%%%%%%%%%%%%%%%%%%%%%%%%%%%%%%%%%%%%%%%%%%%%%%%%
%%%%%%%%%%%%%%%%%%%%%%%%%%%%%%%%%%%%%%%%%%%%%%%%%%%%%%%%%%%%%%%%%%%%%%%%%%%%%%%%%%%%%%%%%%%%%%%%%%%%
%%%%%%%%%%%%%%%%%%%%%%%%%%%%%%%%%%%%%%%%%%%%%%%%%%%%%%%%%%%%%%%%%%%%%%%%%%%%%%%%%%%%%%%%%%%%%%%%%%%%
\newsection{Notations and the paradifferential setup}\label{not_para}

\subsection{Notations}
We shall be using dyadic decompositions of both space and frequency.
For the spatial decomposition, we let $\chi_k$ denote smooth functions
satisfying
\[
1=\sum_{j=0}^\infty \chi_j(x),\quad \supp \chi_0\subset\{|x|\le 2\},
\quad \supp\chi_j\subset \{2^{j-1}<|x|<2^{j+1}\}\text{ for } j\ge 1.
\]
We also set
\[
\chi_{<k} = \sum_{0\le j<k} \chi_j
\]
with the obvious modification for $\chi_{>k}$.  In frequency, we use a
smooth Littlewood-Paley decomposition
\[
1=\sum_{j=-\infty}^\infty S_j(D), \quad \supp
s_j\subset\{2^{j-1}<|\xi|<2^{j+1}\}
\]
and similar notations for $S_{<k}, S_{>k}$ are applied.

We say that a function is frequency localized at frequency $2^k$ if its Fourier transform is supported
in the annulus $\{2^{k-1}<|\xi|<2^{k+1}\}$.  An operator $K$ is said to be frequency localized if
$Kf$ is supported in $\{2^{k-10}<|\xi|<2^{k+10}\}$ for any function $f$ which is frequency localized at
$2^k$.

For $\kappa$ as in \eqref{coeff}, we may choose a positive, slowly varying sequence $\kappa_j\in \ell^1$ satisfying
\begin{equation}
\label{kappa_j}
\sup_{A_j} \la x\ra^2 |\partial_x^2 a(t,x)| + \la x\ra |\partial_x a(t,x)| + |a(t,x)-I_n|\le \kappa_j,
\end{equation}
$$\sum \kappa_j \lesssim \kappa,$$
and
$$|\ln \kappa_j - \ln \kappa_{j-1}|\le 2^{-10}.$$
When the lower order terms are present, we may choose $\kappa_j$ so that each dyadic piece of
\eqref{coeffb} is also controlled similarly.
We may also assume that $M$ in \eqref{coeff.M} is chosen sufficiently large that
$$\sum_{j\ge M} \kappa_j \lesssim \varepsilon.$$
Associated to this slowly varying sequence, we may choose functions $\kappa_k(s)$ with
$$\kappa_0<\kappa_k(s)<2\kappa_0,\quad 0\le s < 2,$$
$$\kappa_j<\kappa_k(s)<2\kappa_j,\quad 2^j<s<2^{j+1},\quad j\ge 1,$$
for $k\ge 0$, 
$$\kappa_k<\kappa_k(s)<2\kappa_k,\quad 0\le s < 2^{-k},$$
$$\kappa_j<\kappa_k(s)<2\kappa_j,\quad 2^j<s<2^{j+1},\quad j\ge -k$$
for $k<0$, 
and
$$|\kappa'_k(s)|\le 2^{-5} s^{-1}\kappa_k(s).$$

%%%%%%%%%%%%%%%%%%%%%%%%%%%%%%%%%%%%%%%%%%%%%%%%%%
\subsection{ Embeddings for the $X$ spaces}
\label{embxs}

Here we prove Lemma~\ref{txdx}.  For the purpose of this section we
can entirely neglect the time variable. 
Let $\psi$ be a smooth, spherically symmetric Schwartz function
with $\psi(0)=1$ which is frequency localized in the unit annulus.
Set
\[
\psi_k(x) = \psi(2^k x).
\]
Given $u \in \tX$, we split it into
\[
u = u^{in} + u^{out}
\]
where
\[
u^{in} = \sum_{k < 0} T_k S_k u
\]
and  $T_k$ is the operator
\[
T_k v = v(t,0) \psi_k(x). 
\]

For frequencies $k > 0$, we
have the dyadic bound
\[
\| \la x \ra^{-1} S_k u\|_{L^2} \lesssim \|S_k u\|_{X_k}  
\]
which we can easily sum over $k$ to obtain
\[
\| \la x \ra^{-1} S_{>0} u\|_{L^2} \lesssim \|u\|_{\dX}.
\]

For frequencies $k < 0$ it  is easy to see that 
\begin{equation}\label{Tkbdd}
\|(1-T_k) S_k u \|_{X_k} \lesssim \|S_k u \|_{X_k}
\end{equation}
follows from the bound
\begin{equation}\label{bernstein}
  \|\chi_{<-k} S_k u\|_{L^2_t L^\infty_x}\lesssim 2^{\frac{n-1}{2}k}\|S_k u\|_{X_k},\quad k\le 0.
\end{equation}
which is a consequence of Bernstein's inequality.

The gain is that $(1-T_k) S_k u(t,0)=0$. This leads to the improved
pointwise bound
\[
 |x|^{-1}|(1-T_k) S_k u| \lesssim 2^{\frac{n+1}2 k} \|S_k u \|_{X_k},\quad |x|<2^{-k}
\]
and further to the improved $L^2$ bound
\begin{equation}
\sup_{j} \| (2^k|x|+2^{-k}|x|^{-1})^\frac12 |x|^{-1}   (1-T_k) S_k
u\|_{L^2(A_j)}  \lesssim 2^{\frac{k}2} \|S_k u \|_{X_k}.
\label{impl2}\end{equation}
Then, by orthogonality with respect 
to spatial dyadic regions,  we can sum up
\[
\| \la x\ra^{-1}   \sum_{k < 0} (1-T_k) S_k u \|_{L^2}^2  \lesssim \|u\|_{\dX}^2 
\]
which combined with the previous high frequency bound yields
\begin{equation}
\|\la x \ra^{-1} u^{out}\|_{L^2} \lesssim \|u\|_{\dX}.
\label{l2uout}\end{equation}

For the terms in $u^{in}$, differentiation yields a $2^k$ factor, and
therefore we can estimate
\begin{equation}
\| u^{in}\|_{\dot H^1} \lesssim \|u\|_{\dX}.
\label{l2uin}\end{equation}
It remains to prove the bounds
\begin{equation}
\| \la x \ra^{-1} v \|_{L^2} \lesssim \|v\|_{L^2(B(0,1))} + \|
v\|_{\dot H^1},
\qquad n=1
\label{n=1emb} 
\end{equation}
respectively
\begin{equation}
\| \la x \ra^{-1} (\ln (1+ \la x \ra))^{-1}
 v \|_{L^2} \lesssim \|v\|_{L^2(B(0,1))} + \|
v\|_{\dot H^1},
\qquad n=2.
\label{n=2emb} 
\end{equation}

Due to the first factor in the right of both estimates, we may without loss of generality take
$v$ to vanish in $B(0,1/2)$.
For \eqref{n=1emb} we integrate
\[
2 \int_{1/2}^R  x^{-1} v v_x dx = \int_{1/2}^R  x^{-2} v^2 dx  + R^{-1}
v^2(R).
\]
Using Cauchy-Schwarz the conclusion follows.

For \eqref{n=2emb} we argue in a similar fashion.
We have
%\[
%\begin{split}
\begin{multline*}
2 \int_{B_R \setminus B_{1/2}}  |x|^{-2} (\ln (2+|x|^2))^{-1} v x \nabla v dx = 
\int_{B_R \setminus B_{1/2}}  (2+|x|^2)^{-1} (\ln (2+|x|^2))^{-2}  v^2 dx 
\\ + \int_{\partial B_R}|x|^{-1} (\ln (2+|x|^2))^{-1} v^2   d\sigma
%- \int_{\partial B_1}  |x|^{-1} (\ln (2+x^2))^{-1} v^2 d\sigma
\end{multline*}
%\end{split}
%\]
and conclude again by Cauchy-Schwarz. The lemma is proved. \qed

On a related note, we include here another  result which
simplifies the type of local error terms we allow in the non-trapping
case.

\begin{lemma}
Let $n \geq 1$ and $R > 0$. Then for each $\e > 0$ there is 
$m_\epsilon > 0$ and $c_\epsilon > 0$
so that
\begin{equation}\label{dXerror}
\| \la x \ra^{-\frac32} u\|_{L^2} \leq \e \|u\|_{\dX} + c_\epsilon
\|S_{<m_\e}u\|_{L^2(\{|x| < R\})}.
\end{equation}
\end{lemma}
 
\begin{proof}
Frequencies in $u$ which are large enough can be estimated solely by
the first term on the right.  It remains to show that for large $m$ we
have
\[
\| \la x \ra^{-\frac32} S_{<m} u\|_{L^2} \leq \e \| S_{<m} u\|_{\dX} + c_{\epsilon,m}
\|S_{<m} u\|_{L^2(\{|x| < R\})}.
\]
For large $x$ the left hand side can also be estimated solely by the
first term on the right.  It remains to show that for large $m,k$ we
have
\[
\| \la x \ra^{-\frac32} \chi_{<k} S_{<m} u\|_{L^2} \leq \| \la x \ra^{-\frac32} \chi_{>k} S_{<m} u\|_{L^2} + c_{k,m}
\|S_{<m} u\|_{L^2(\{|x| < R\})}.
\]
We argue by contradiction. Suppose this is false. Then there exists
a sequence $u_j \in \dX$ so that 
\[
\| \la x \ra^{-\frac32} \chi_{<k} S_{<m} u_j\|_{L^2} = 1, \qquad
\| \la x \ra^{-\frac32} \chi_{>k} S_{<m} u_j\|_{L^2} < 1, \qquad 
\|S_{<m} u_j\|_{L^2(\{|x| < R\})} \to 0.
\]
The functions $ \la x \ra^{-\frac32}  S_{<m} u_j$ are uniformly
bounded in all Sobolev spaces $H^{N}(\R^n)$; therefore on a subsequence
we have uniform convergence on compact sets,
\[
 S_{<m} u_j \to u.
\]
Then the function $u$ 
satisfies
\[
\| \la x \ra^{-\frac32} \chi_{<k} u\|_{L^2} = 1, \qquad
\| \la x \ra^{-\frac32} \chi_{>k}  u\|_{L^2} < 1, \qquad 
\|u\|_{L^2(\{|x| < R\})} = 0.
\]
But $u$ is also frequency localized in $|\xi| < 2^{m+1}$ and
is therefore analytic. Then the last condition above implies $u = 0$ which is a 
contradiction.
\end{proof}

%%%%%%%%%%%%%%%%%%%%%%%%%%%%%%%%%%%%%%%%%%%%%%%%%
\subsection{Paradifferential calculus}

Here, we seek to frequency localize the coefficients of $P$.
A similar argument is present in  \cite{T}, where for solutions at 
frequency $2^k$ the coefficients are localized at frequency
\[
|\xi| \ll 2^{k/2} \la x\ra^{-1/2}.
\]
Such a strong localization was essential there in order to carry out
the parametrix construction. Here we are able to keep the setup
simpler and use a classical paradifferential construction, where for
solutions at frequency $2^k$ the coefficients are localized at
frequency below $ 2^k$. For a fixed frequency scale $2^k$, we set
\begin{align*}
a^{ij}_{(k)} = S_{<k-4} a^{ij},
\end{align*}
and we define the associated mollified operators
\[
A_{(k)}= D_i a^{ij}_{(k)}D_j.
\]
It is easy to verify that the mollified coefficients $a^{ij}_{(k)}$
satisfy the bounds
\begin{equation}\label{coeffak.greater} 
\begin{split}
|\partial^\alpha (a^{ij}_{(k)}-I_n)|&\lesssim \kappa_k(|x|)
\la x \ra ^{-|\alpha|},\qquad 
|\alpha|\le 2, \quad  k>0
\\
|\partial^\alpha (a^{ij}_{(k)}-I_n)|&\lesssim \kappa_k(|x|)
{2^{|\alpha| k}}{\la 2^k x \ra^{-|\alpha|}},\quad 
|\alpha|\le 2, \quad  k\leq 0.
\end{split}
\end{equation}

The next proposition will be used to pass back and forth between $A_{(k)}$
and $A$. We first define
\[
\tilde{A}=\sum_k A_{(k)}S_k.
\]

\begin{prop}
\label{lemma.A.to.Ak}
Assume that the coefficients $a^{ij}$ satisfy \eqref{coeff}, and that
$b=0$, $c=0$.  Then
\begin{equation}
  \label{aak}
  \sum_{k} 2^{-k} \|S_k (A-A_{(k)}) u\|_{X_k'}^2 \lesssim 
\kappa^2 \|u\|_{\dX}^2, 
\end{equation}
\begin{equation}\label{amta}
\|(A-\tilde{A})u\|_{\dX'}\lesssim \kappa \|u\|_{\dX},
\end{equation}
\begin{equation}
  \label{comak}
2^{-k} \|[A_{(k)},S_k] u\|_{X'_k}\lesssim \kappa \|u\|_{X_k}.
\end{equation}
\end{prop}
\begin{proof}[ Proof of Lemma \ref{lemma.A.to.Ak}:]
We begin by writing
\[
S_k(A- A_{(k)})=A_k^{med}+A_k^{high}
\]
with
\begin{align*}
A_k^{med}&= 
 \sum_{l=k-4}^{k+4} \sum_{m=-\infty}^{k+8} S_k D_i(S_l a^{ij})
D_j S_m \\
A_k^{high}&=  \sum_{l>k+4} \sum_{m=l-4}^{l+4} 
S_k D_i (S_la^{ij})D_j S_m.
\end{align*}

For $A_k^{med}$ we take $l=k \geq m$ for simplicity; then it suffices 
 to establish the off-diagonal decay 
\begin{equation}
  \label{dplow3.small}
\|S_k D_i(S_k a^{ij}D_j S_m v)\|_{X'_k}\lesssim \kappa
2^m \|S_m v\|_{X_m}.
\end{equation}

If $k \geq m \geq 0$ then we have
\[
\begin{split}
\|S_k D_i(S_k a^{ij}  D_j S_m v)\|_{X'_k}
&\ \lesssim 2^k \|S_k a^{ij} D_j S_m v \|_{X'_k}
\\ &\ \lesssim \kappa 2^{-k} \| \la x\ra^{-2} D_j S_m v \|_{X'_k}
\\ &\ \lesssim \kappa 2^{-k} \| D_j S_m v \|_{X_m}
\\ &\ \lesssim \kappa 2^{m-k} \|S_m v \|_{X_m}.
\end{split}
\]

If $k \geq 0 > m$ then we have two spatial scales to deal with, namely
$1$ and $2^{-m}$.  To separate them we use the cutoff function
$\chi_{<-m}$. For contributions corresponding to large $x$ we estimate
\[
\begin{split}
\|S_k D_i(S_k a^{ij} \chi_{\ge -m} D_j S_m v)\|_{X'_k}
&\ \lesssim 2^k \|S_k a^{ij} \chi_{\ge -m} D_j S_m v \|_{X'_k}
\\ &\ \lesssim \kappa 2^{-k} \| |x|^{-2}   \chi_{\ge -m} D_j S_m v \|_{X'_k}
\\ &\ \lesssim \kappa 2^{m-k} \| D_j S_m v \|_{X_m}
\\ &\ \lesssim \kappa 2^{2m-k} \| S_m v \|_{X_m}.
\end{split}
\]
For contributions  corresponding to small $x$, we first note that by 
Bernstein's inequality, see \eqref{bernstein}, we have 
\begin{equation}\label{bern.2}
\| D_j S_m v\|_{L^2_t L_x^\infty(A_{\leq -m})} \leq 2^{\frac{n+1}2 m} \|S_m v\|_{X_m}.
\end{equation}
Then
\[
\begin{split}
\|S_k D_i(S_k a^{ij} \chi_{<-m} D_j S_m v)\|_{X'_k}
&\ \lesssim 2^k \|S_k a^{ij} \chi_{<-m} D_j S_m v\|_{X'_k}
\\
&\ \lesssim 2^{-k} 2^{\frac{n+1}2 m} \|\la x \ra^{-2}  \chi_{<-m}
\kappa(|x|)\|_{(X_k^0)'} \|S_m v\|_{X_m}
\\
&\ \lesssim \kappa 2^{-k}2^{\frac{n+1}2 m} \max\{ 1, 2^{\frac{3-n}{2} m} \} \|S_m v\|_{X_m}
\\
&\ \lesssim \kappa 2^{-k}  \max\{ 2^{\frac{n+1}2 m}, 2^{2 m} \} \|S_m v\|_{X_m}
\end{split}
\]
where $(X_k^0)'$ is the spatial part of the $X_k'$ norm, i.e. $X_k' = L^2_t (X_k^0)'$.

Finally if $ 0 >k \geq  m$ then the spatial scales  are $2^{-k}$ and
$2^{-m}$, and we separate them using the cutoff function
$\chi_{<-m}$. The exterior part is exactly as in the previous case. For
the interior part we use again \eqref{bern.2} 
to compute
\[
\begin{split}
\|S_k D_i(S_k a^{ij} \chi_{<-m} D_j S_m v)\|_{X'_k}
&\ \lesssim 2^k \|S_k a^{ij} \chi_{<-m} D_j S_m v\|_{X'_k}
\\
&\ \lesssim 2^{k} 2^{\frac{n+1}2 m} \|\la 2^k x \ra^{-2}  \chi_{<-m}
\kappa(|x|)\|_{(X_k^0)'} \|S_mv\|_{X_m}
\\
&\ \lesssim \kappa 2^{k}2^{\frac{n+1}2 m} 
\max\{ 2^{-\frac{n+1}2 k}, 2^{-2k} 
2^{\frac{3-n}{2} m} \} \|S_m v\|_{X_m}
\\
&\ \lesssim \max\{ 2^{\frac{1-n}2 k}  2^{\frac{n+1}2 m}, 2^{-k} 2^{2 m} \} \|S_m v\|_{X_m}.
\end{split}
\]
Hence \eqref{dplow3.small} is proved, which by summation yields the bound
\eqref{aak} for $A_k^{med}$.  The bound for $A_k^{high}$ follows from
 summation of \eqref{dplow3.small} in a duality argument.

We note that in all cases there is some room to spare in the
estimates. This shows that our hypothesis is too strong for this
lemma. Indeed, one could prove it without using at all the bound on
the second derivatives of the coefficients.

The bound \eqref{amta} follows by duality from \eqref{aak}.  The proof
of \eqref{comak}, as in \cite{T}, follows from the $|\alpha|=1$ case
of \eqref{coeffak.greater}.  
\end{proof}

%%%%%%%%%%%%%%%%%%%%%%%%%%%%%%%%%%%%%%%%%

The next  proposition   allows us to treat lower order terms perturbatively 
in most of our results.

\begin{prop}\label{lemma.lower.order}
  a)  Assume that  $b,c$ satisfy \eqref{coeffb} and \eqref{coeffc}.
Then
 \begin{equation}\label{bc}
 \|(b^i D_i + D_i b^i +c)u\|_{\tX'}\lesssim \kappa \|u\|_{\tX}.
  \end{equation}
 
b) Assume that  $b$ satisfies \eqref{coeffb} and $\div \ b=0$.
Then
 \begin{equation}\label{bnoc}
 \|(b^i D_i + D_i b^i)u\|_{\dX'}\lesssim \kappa \|u\|_{\dX}.
\end{equation}

\end{prop}

\begin{proof}
This proof parallels a similar argument in \cite{T}. However in there 
only dimensions $n \geq 3$ are considered, and the bound
\eqref{coeffc} is stronger to include the full gradient of $b$.
Thus we provide a complete proof here. We consider two cases,
the first of which is similar to  \cite{T}, while  the second
requires a new argument.

{\bf Case 1: The estimate \eqref{bc} for $n \geq 3$ and \eqref{bnoc}
  for $n=1,2$.}
The estimate for the $c$ term is straightforward since, by \eqref{coeffc}, 
\[ \la cu, v\ra \lesssim \kappa \|\la x\ra^{-1} u\|_{L^2_{t,x}} \|\la x\ra^{-1} v\|_{L^2_{t,x}}
\lesssim \kappa \|u\|_{\tX} \|v\|_{\tX}.\]
  
For the $b$ term, we consider a
paradifferential decomposition,
\begin{multline}\label{trichotomy}
(b^i D_i + D_i b^i)u=\ \sum_{k}  (S_{<k}  b^i D_i+ D_i S_{<k}  b^i
)S_k u 
\\  \ + \sum_k (S_k b^i D_i + D_i S_{k}  b^i)S_k u 
 \ + \sum_{k} (S_{>k}b^i D_i + D_i S_{>k}  b^i)S_k u.
\end{multline}

The frequency localization is preserved in the first term; therefore 
it suffices to verify that
\[
\| (S_{<k}  b^i D_i+ D_i S_{<k}  b^i)S_k u \|_{X'_{k}} \lesssim \kappa 2^{k}\|S_k u\|_{X_k}.
\]
The derivative yields a factor of $2^k$, and we are left with proving
that
\[
\| S_{<k}  b^i v\|_{X'_{k}} \lesssim \kappa \| v\|_{X_k}.
\]
This in turn follows from the pointwise bound
\[
|S_{<k}  b^i | \lesssim 
\begin{cases}
\kappa_k(|x|)\la x\ra^{-1}  ,\quad k\ge 0,\\
\max\Bigl\{ 2^k \kappa_k(|x|)\la 2^k x\ra^{-1}, \kappa 2^k\la 2^k x\ra^{-2}\Bigr\},\quad k <0\end{cases}
\]
which is easy to obtain. The second term on the second line above is
only needed in the worst case $n=1$.

The remaining two terms in \eqref{trichotomy} are dual. Hence it
suffices to consider the last one. We want the derivative to go to the
low frequency; therefore we rewrite it in the form
\begin{equation}
\sum_{k}  2 S_{>k}b^i  D_i  S_k u   -i S_{>k} \div\ b \ S_k u.
\label{hl}
\end{equation}
We consider the two terms separately. The second one occurs only in 
the case of \eqref{bc} but the first one occurs also in \eqref{bnoc}.
So we need to show that
\[
\| \sum_{k}  S_{>k}b^i  D_i  S_k u  \|_{\dX'} \lesssim \kappa \| u\|_{\dX}.
\]
This will follow from the dyadic estimates
\[
\| S_{m}b^i  S_k u  \|_{X'_m} \lesssim \kappa   \| S_k
u\|_{X_k}, \qquad m > k.
\]
Given the pointwise bound on $S_m b^i$, this reduces to
\[
\| S_k u  \|_{X_m} \lesssim     \| S_k
u\|_{X_k}.
\]
For $|x| > \max\{2^{-k},1\}$ this is trivial. For smaller $x$
we use \eqref{bernstein},
%\begin{equation}
%\| \chi_{<-k}  S_k u\|_{L^2_tL^\infty_x} \lesssim 2^{\frac{n-1}2 k} \|S_k u\|_{X_k},
%\quad k\le 0,
%\label{bernstein.xk}\end{equation}
and the conclusion is obtained by a direct computation.

It remains to consider the second term in \eqref{hl}, for which we 
want to show that in dimension $n \geq 3$
\begin{equation}
\| \sum_{k}  S_{>k} \div\ b \ S_k u \|_{\tX'} \lesssim \kappa \| u\|_{\tX}.
\label{hhll}\end{equation}

For this we establish again off-diagonal decay,
\begin{equation}
\| S_{m} \div\ b \ S_k u \|_{X'_m} \lesssim \kappa (m-k) 2^k  \| S_k u\|_{X_k},
\qquad m > k.
\label{odhhll}\end{equation}
This follows from the pointwise bounds
\[
| S_{m} \div\ b | \leq \kappa 2^{2m}\la 2^m x\ra^{-2}, \qquad m < 0
\]
\[
| S_{m} \div\ b | \leq \kappa \la  x\ra^{-2}, \qquad m \geq 0.
\]
We consider the worst case $0 > m > k$ and leave the rest for the
reader. We use $\chi_{<-k}$ to separate small and large values of $x$.
For large $x$ we have
\[
\| \chi_{>-k} S_{m} \div\ b \ S_k u \|_{X'_m} \lesssim 
\kappa \| |x|^{-2} \chi_{>-k}  \ S_k u \|_{X'_k} \lesssim 
\kappa 2^k  \| S_k u\|_{X_k}.
\]
For small $x$ we use \eqref{bernstein} instead,
\[
\| \chi_{<-k} S_{m} \div\ b \ S_k u \|_{X'_m} \lesssim 
 \kappa 2^{2m} 2^{\frac{n-1}2 k}  \| \chi_{<-k} \la 2^m x\ra^{-2}\|_{(X_m^0)'} \|S_k u\|_{X_k}
\lesssim   \kappa  2^{ k}  \|S_k u\|_{X_k}.
\]
The last computation above is accurate if $n \geq 4$. In dimension
$n=3$ we encounter a harmless additional logarithmic factor $|m-k|$.
However if $n=1,2$ then the above off-diagonal decay can no longer 
be obtained.

{\bf Case 2: The estimate \eqref{bc} in dimension $n=1,2$.}
The $c$ term is again easy to deal with. We write the 
estimate for $b$ in a symmetric way,
\[
|\la (b^i D_i + D_i b^i) u,v \ra | \lesssim \kappa \|u\|_{\tX} \|v\|_{\tX}.
\]
We use the decomposition in Section~\ref{embxs},
\[
u = u^{in} + u^{out}, \qquad v =  v^{in} + v^{out}.
\]
We consider first the expression 
\[
\la (b^i D_i + D_i b^i) u^{out},v^{out} \ra.
\]
For this we can take advantage of the improved $L^2$ bound
\eqref{impl2} to carry out the same computation as in dimension 
$n \geq 3$, establishing off-diagonal decay. Precisely,
the difference arises in the proof of \eqref{odhhll}, whose
replacement is 
\begin{equation}
\| S_{m} \div\ b \ (1-T_k) S_k u \|_{X'_m} \lesssim \kappa (m-k)2^k  \| S_k u\|_{X_k},
\qquad m > k.
\label{odhhll.new}\end{equation}

Consider now one of the cross terms, 
\[
\la (b^i D_i + D_i b^i) u^{in},v^{out} \ra = \la (2 b^i D_i  -i \div  b) u^{in},v^{out} \ra.
\]
The proof for the other cross term will follow similarly.
For the $\div \ b$ term we use the $L^2$ bound for both $u^{in}$ and
$v^{out}$, as in the case of $c$. For the rest we use \eqref{l2uin}
and \eqref{l2uout} to estimate
\[
|\la  b^i D_i  u^{in},v^{out} \ra| \lesssim \|u^{in}\|_{\dot H^1} \| b
v^{out}\|_{L^2} \lesssim \|u\|_{\dX} \|v\|_{\dX}.
\]

 Finally, consider the last term
 \[
\la (b^i D_i + D_i b^i) u^{in},v^{in} \ra.
\]
In dimension $n=1$, we can easily estimate it by
\[
|\la (b^i D_i + D_i b^i) u^{in},v^{in} \ra| \lesssim \|u^{in}\|_{\dot
  H^1} \| \la x \ra^{-1} v^{in}\|_{L^2} + \|v^{in}\|_{\dot
  H^1} \| \la x\ra^{-1} u^{in}\|_{L^2} \lesssim \|u\|_{\tX} \|v\|_{\tX}.
\]
This argument fails for $n=2$ due to the logarithmic factor in the
$L^2$ weights. Instead we will take advantage of the
spherical symmetry of both $u^{in}$ and $v^{in}$.

In polar coordinates we write
\[
b^i D_i = b^r D_r + r^{-1} b^\theta D_\theta 
\] 
and
\[
\div \ b = \partial_r b^r + r^{-1} b^r + r^{-1} \partial_\theta b^\theta.
\]
For a function $b(r,\theta)$, we denote $\bar b(r)$ its spherical
average. By spherical symmetry, we compute
\[
\la b^i D_i  u^{in},v^{in} \ra = \la  (b^r D_r + r^{-1} b^\theta D_\theta ) u^{in},v^{in} \ra
=  \la D_r  u^{in},  \bar{b^r} v^{in} \ra.
\]
Then we can estimate
\[
|\la (b^i D_i + D_i b^i) u^{in},v^{in} \ra| \lesssim \|u^{in}\|_{\dot
  H^1} \|  \bar{b^r} v^{in}\|_{L^2} + \|v^{in}\|_{\dot H^1} \|\bar{b^r}u^{in}\|_{L^2}
\lesssim \|u\|_{\tX} \|v\|_{\tX}
\]
provided we are able to establish the improved bound
\begin{equation}
| \bar{b^r}(r)| \lesssim \la r \ra^{-1}(\ln (2+r))^{-1}.
\label{bbr}\end{equation}

For this we take spherical averages in the divergence equation
to obtain
\[
 \partial_r \bar b^r + r^{-1} \bar b^r = \overline{\div \ b}.
\]
At infinity we have $b(r) = o(r^{-1})$. Integrating from infinity
we obtain
\[
 \bar b^r(r) = \int_{r}^\infty \frac{s}{r}\ \overline{\div \ b}(s) ds.
\]
Hence
\[
| \bar b^r(r) | \lesssim \int_{r}^\infty \frac{s}{r} (1+s)^{-2}
(\ln(2+s))^{-2} ds
\]
and \eqref{bbr} follows.
\end{proof}

%%%%%%%%%%%%%%%%%%%%%%%%%%%%%%%%%%%%%%%%%%%%%%%%%%%%%%%%%%%%%%%%%%%%%%%%%%%%%%%%%%%%%%%%%%%%%%%%%%%%
%%%%%%%%%%%%%%%%%%%%%%%%%%%%%%%%%%%%%%%%%%%%%%%%%%%%%%%%%%%%%%%%%%%%%%%%%%%%%%%%%%%%%%%%%%%%%%%%%%%%
%%%%%%%%%%%%%%%%%%%%%%%%%%%%%%%%%%%%%%%%%%%%%%%%%%%%%%%%%%%%%%%%%%%%%%%%%%%%%%%%%%%%%%%%%%%%%%%%%%%%
\bigskip
\newsection{Local smoothing estimates}\label{smoothing.section}

In this section we prove our main local  smoothing estimates, 
first in the exterior region and then in the non-trapping case.

\subsection{The high dimensional case $n \geq 3$: Proof of Theorem~\ref{main.ls.theorem}}
The proof uses  energy estimates and the positive commutator
method. This turns out to be rather delicate. The difficulty 
is that the trapping region acts essentially as a black box,
where the energy is conserved but little else is known. 
Hence all the local smoothing information has to be estimated 
starting from infinity along rays of the Hamilton flow which are 
 incoming either  forward or backward in time.

We begin with the energy estimate. This is standard
if the right hand side is in $L^1_t L^2_x$, but we would like to allow 
the right hand side to be in the dual smoothing space as well.
% An added difficulty occurs in low dimension $n=1,2$, where 
% the local smoothing space is a quotient space, and therefore
% is not stable with respect to multiplication by bump functions.

\begin{proposition}
Let $u$ solve the equation
\begin{equation}
D_t +A u = f_1 + f_2, \qquad u(0) = u_0
\end{equation}
in the time interval $[0,T]$. Then we have 
\begin{equation}
\| u\|_{L^\infty_t L^2_x}^2 \lesssim \| u_0\|_{L^2}^2 + \| f_1\|_{L^1_t L^2_x}^2
+ \| u\|_{\tX_e} \|f_2\|_{\tX'_e}. 
\label{eest}\end{equation}
\end{proposition}
\begin{proof}
The proof is straightforward.  We compute
\[
\frac{d}{dt} \frac12 \|u(t)\|_{L^2}^2 = \Im \langle u, f_1+f_2 \rangle.
\]
Hence for each $t \in [0,T]$ we have
\[
\| u(t)\|_{L^2}^2 \lesssim \|u(0)\|_{L^2}^2 + \|u\|_{L^\infty_t L^2_x}
\| f_1\|_{L^1_t L^2_x} +  \| u\|_{\tX_e} \|f_2\|_{\tX'_e}.
\]
We take the supremum over $t$ on the left and use bootstrapping
for the second term on the right.  The conclusion follows.
\end{proof}

To prove \eqref{lsext} we need a complementary estimate,
namely  
\begin{equation}
  \| \rho u\|_{\tX}^2 \lesssim \| u\|_{L^\infty_t L^2_x}^2 + 
  \| f_1\|_{L^1_t L^2_x}^2 
  +  \|\rho f_2\|_{\tX'}^2 + \| \la x \ra^{-2}  u\|_{L^2_{t,x}}^2. 
\label{lsest}\end{equation}
Given \eqref{eest} and \eqref{lsest}, the bound  \eqref{lsext}
is obtained by bootstrapping, with some careful balancing 
of constants. 

It remains to prove \eqref{lsest}. 
We will use a positive commutator method.  We shall assume that $b=0$ and $c=0$.
For a self-adjoint operator
$Q$, we have
\[
2\Im\la A u, Qu\ra = \la Cu,u\ra
\]
where
\[
C=i[A,Q].
\]
As a consequence of this, we see that
\[
\frac{d}{dt}\la u, Qu\ra = -2\Im\la (D_t+A)u, Qu\ra + \la
Cu,u\ra.
\]
Taking this into account, the estimate \eqref{lsest} is an immediate
consequence of the following lemma.

\begin{proposition} \label{Qprop}
 There is a family $\cal Q$ of bounded self-adjoint operators $Q_\rho$ with 
the following properties:

(i) $L^2$ boundedness,
\[
\|Q_\rho\|_{L^2 \to L^2} \lesssim  1
\]

(ii) $\tX$ boundedness,
\[
|\langle Q_\rho u, f \rangle | \lesssim \| \rho f\|_{\tX'} \| \rho u\|_{\tX}
\]

(iii) Positive commutator,
\[
\sup_{Q_\rho \in \cal Q} \langle Cu,u \rangle \geq  c_1 \|\rho u\|_{\tX}^2 -
c_2 \|\la x\ra^{-2}  u\|_{L^2_{t,x}}^2.
\]
\end{proposition}

  We first note that the condition (ii) shows that $Q_\rho u$ is supported
  in $\{|x| > 2^M\}$ and depends only on the values of $u$ in the same
  region. Hence for the purpose of this proof we can modify the
  operator $A$ arbitrarily in the inner region $\{|x| < 2^M\}$. In
  particular we can improve the constant $\kappa$ in \eqref{coeff} to
  the extent that \eqref{coeff.M} holds globally. Similarly, we can
  assume without any restriction in generality that $u = 0$ in $\{|x|
  < 2^M\}$.  

Using (ii), we may argue similarly and assume that \eqref{coeffb.M} and
\eqref{coeffc.M} hold globally if lower order terms are present. 
The estimate \eqref{bc} then justifies
 neglecting the 
lower order terms in $A$.  I.e., we may assume that $b=0$, $c=0$.

\begin{proof}
The main step in the proof of the proposition is to construct
some frequency localized versions of the operator $Q_\rho $. Precisely,
for each $k \in \Z$ we produce a family ${\cal Q}_k$ 
of operators $Q_k$, which we later use to construct $Q_\rho $.
We consider two cases, depending on whether $k$ is positive or
negative.

We first introduce some variants of the spaces $X_k$.  Let $k \in \Z$ and
$k^-=\frac{|k|-k}{2}$ be its negative part. For any positive, slowly varying
sequence $(\alpha_m)|_{m\geq k^-}$ with
\[
\sum_{k \ge k^-}\alpha_j =1, \qquad \alpha_{k^-} \approx 1,
\]
we define the space $X_{k,\alpha}$ with norm
\begin{align*}
  \|u\|_{X_{k,\alpha}}^2 &= 2^{-k^-} \|u\|^2_{L^2(A_{\leq k^-})} +
  \sum_{j>k^-} \alpha_j\| |x|^{-1/2} u\|^2_{L^2(A_j)}.
%   \quad &k< 0,\\
%   \|u\|_{\tX_{k,\alpha}}^2 &= \|u\|^2_{L^2(A_{<0})} + \sum_{j>0}
%   \alpha_j \|\la x\ra^{-1/2} u\|^2_{L^2(A_j)}, \quad &k\ge 0.
\end{align*}
% Associated to these spaces, we have the dual spaces
% \begin{align*}
% \|f\|^2_{\tX'_{k,\alpha}}&= 2^{\frac{k^-}2} \|f\|^2_{L^2(A_{<k^-})} + \sum_{j>k^-} \alpha_j^{-1}
% \||x|^{1/2} f\|^2_{L^2(A_j)}
% % ,\quad &k<0,\\
% % \|u\|^2_{\tX'_{k,\alpha}}&= \|u\|^2_{L^2(A_{<0})} + \sum_{j>0} \alpha_j^{-1}
% % \||x|^{1/2} u\|^2_{L^2(A_j)},\quad &k\ge 0.
% \end{align*}
Then our low frequency result has the form
\begin{lemma} \label{lemma.low.freq} 
Let $n \geq 1$ and $k<0$. Then for any slowly varying sequence $(\alpha_m)$ with
$\alpha_{-k} \approx 1$ and $\sum_{m\ge -k}\alpha_m=1$, there
  is a self-adjoint operator $Q_k$ so that
\begin{align}
\|Q_k u\|_{L^2}&\lesssim  \|u\|_{L^2},\label{Q.L2.low}\\
\|Q_k u\|_{X_{k,\alpha}}&\lesssim  \|u\|_{X_{k,\alpha}},\label{Q.X.low}\\
\la C_k u, u\ra &\gtrsim 2^{k} \|u\|^2_{X_{k,\alpha}}, \qquad 
C_k = i [ A_{(k)}, Q_k] 
\label{C.low}
\end{align}
for all functions $u$ frequency localized at frequency $2^k$.
\end{lemma}

\begin{proof}
  We argue exactly as in \cite[Lemma 9]{T}. The only difference is
  that here we work with the operator $A_{(k)}$ whose coefficients have
  less regularity, but this turns out to be nonessential.
 
  We first increase the sequence $(\alpha_m)$ so that
\begin{equation}
\label{new.alpha}
\begin{cases}
(\alpha_m)\text{ remains slowly varying,}\\
\alpha_m = 1\text{ for } m\le -k \\
\displaystyle \sum_{m>-k} \alpha_m \approx 1,\\
\kappa_m \le \epsilon \alpha_{m}\text{ for }  m > -k.
\end{cases}
\end{equation}
To this slowly varying sequence we may associate a slowly varying
function $\alpha(s)$ with
\[
\alpha(s)\approx \alpha_m,\quad s\approx 2^{m+k}.
\]

We construct an even smooth  symbol $\phi$  of order $-1$ satisfying
\begin{align}
\phi(s)&\approx \la s\ra^{-1},\quad s>0 \label{phi1}\\
\phi(s)+s\phi'(s)&\approx \frac{\alpha(s)}{\la s\ra},\quad
s>0.\label{phi2} 	
\end{align}
We notice that the radial function $ S_{<10}(D)\phi(|x|)$ satisfies
the same estimates; therefore without any restriction in generality we
assume that $\phi(|x|)$ is frequency localized in $|\xi| < 2^{10}$.

We now define the self-adjoint multiplier
\[
Q_k(x,D)= \delta(Dx\phi( 2^k \delta |x|)+\phi(2^k \delta |x|)xD).
\]
For small $\delta$ this takes frequency $2^k$ functions to frequency
$2^k$ functions. The first property \eqref{Q.L2.low} follows immediately.
\begin{comment}
\begin{com}
  Indeed, this is easy since we are in $L^2$ and $\phi_k$ preserves
  the frequency support.  Thus, the $D$ introduces the $2^k$, and the
  $\delta |x| \phi_k \le 1$ by the first property of $\phi$ above.
\end{com}
\end{comment}
The estimate \eqref{Q.X.low} is also straightforward as the weight in
the $X_{k,\alpha}$ norm is slowly varying on the dyadic scale.  It
remains to prove \eqref{C.low} for which we begin by computing the
commutator
\begin{equation}
\label{C}
\begin{split}
C_k =&\  4\delta D_i \phi(2^k\delta |x|)a^{ij}_{(k)}D_j 
\\ &\ + 2^{k+1}
\delta^2 \Bigl(Dx |x|^{-1}\phi'(2^k \delta |x|)
x_i a^{ij}_{(k)} D_j + D_i a^{ij}_{(k)} x_j  |x|^{-1}\phi'(2^k \delta |x|)xD\Bigr)
\\ &\ -2\delta D_i\phi(2^k \delta |x|)(x_l\partial_l a^{ij}_{(k)})D_j + \partial_i(a^{ij}_{(k)}(\partial_j \partial(
\delta x\phi(2^k \delta |x|)))).
\end{split}
\end{equation}
The positive contribution comes from the first two terms. Replacing
$a^{ij}_{(k)}$ by the identity leaves us with the principal part
\[
C_k^0=4\delta D \phi(2^k\delta |x|)D + 4\delta D \frac{x}{|x|} 2^k 
\delta |x|\phi'(2^k\delta |x|)\frac{x}{|x|}D
\]
which by \eqref{phi2} satisfies
\[
\la C_k^0 u,u\ra \ge 4\delta \la (\phi(2^k \delta |x|)+2^k \delta
|x|\phi'(2^k \delta |x|))\nabla u,\nabla u\ra
\gtrsim \delta 2^{2k}\Bigl\la \frac{\alpha(2^k \delta |x|)}{\la 2^k \delta
  x\ra}u,u\Bigr\ra.
\]
Since $a^{ij}_{(k)}(x)-\delta^{ij} = O(\kappa_k(|x|))$, the error we produce
by substituting $a^{ij}_{(k)}$ by the identity has size
\[
\delta 2^{2k} \left\langle \frac{\kappa_k(|x|)}{\la 2^k \delta x\ra}u,u\right\ra.
\]
It remains to examine the last two terms in $C_k$.  Using
\eqref{coeffak.greater}, we see that
\[
|\delta \phi(2^k \delta |x|)(x_l\partial_l a^{ij}_{(k)})|\lesssim
\frac{\delta\kappa_k(|x|)}{\la 2^k \delta x\ra}.
\]
So, the third term yields an error similar to the above one.

Finally, 
\[
|\partial_i(a^{ij}_{(k)}(\partial_j\partial(\delta x\phi(2^k \delta |x|))))|\lesssim
\frac{\delta^3 2^{2k}}{\la 2^k\delta x\ra^3}\lesssim
\frac{\delta^3 2^{2k} \alpha(2^k \delta |x|)}{\la 2^k \delta x\ra},
\]
which yields
\[
\la\partial_i(a^{ij}_{(k)}(\partial_j\partial\delta x\phi(2^k \delta |x|)))u,u\ra
\lesssim \delta^3 2^{2k} \Bigl\la \frac{\alpha(2^k \delta |x|)}{\la 2^k \delta
  x\ra} u,u\Bigr\ra.
\]

Summing up, we have proved that
\begin{equation}
\langle C_ku,u\rangle \geq c_1 \delta 2^{2k}\Bigl\la \frac{\alpha(2^k \delta |x|)}{\la 2^k \delta
  x\ra}u,u\Bigr\ra\!  - c_2 \delta^3 2^{2k}
\Bigl\la \frac{\alpha(2^k \delta |x|)}{\la 2^k \delta
  x\ra} u,u\Bigr\ra \!- c_3 \delta 2^{2k} \left\langle \frac{\kappa_k(|x|)}{\la 2^k \delta x\ra}u,u\right\ra.
\label{lowcom}\end{equation}
In order to absorb the second term into the first we need to know that
$\delta$ is sufficiently small. This determines the choice of $\delta$
as a small universal constant. In order to absorb the third term into
the first we use the last part of \eqref{new.alpha} and the fact that 
$\alpha$ is slowly varying on the dyadic scale to estimate
\[
\kappa(|x|) \lesssim \epsilon \alpha(2^k |x|) \lesssim \delta^{-1}
\epsilon \alpha(2^k \delta |x|).
\]
Thus the third term is negligible if $\epsilon \ll \delta$. This
determines the choice of $\epsilon$ in \eqref{coeff.M},
\eqref{coeffb.M} and \eqref{coeffc.M}.
\end{proof}

We continue with the result for high frequencies.

\begin{lemma}
\label{lemma.high.freq}
Let $n \geq 1$ and $k \geq 0$. Then for any  sequence $(\alpha_m)$ with
$\alpha_0=1$ and $\sum_{m\ge 0} \alpha_m=1$  there
is a self-adjoint operator $Q_k$ so that
\begin{align}
\|Q_k u\|_{L^2} &\lesssim \|u\|_{L^2},\label{Q.L2.high}\\
\|Q_k u\|_{X_{k,\alpha}} 
&\lesssim  \|u\|_{X_{k,\alpha}},\label{Q.X.high}\\
\la C_k u,u\ra &\gtrsim 2^{k}\| u\|^2_{X_{k,\alpha}}
, \qquad 
C_k = i [ A_{(k)}, Q_k],
\label{C.high}\\
2\Im\la[A_{(k)},\rho_{<k}]u,Q_k \rho_{<k} u\ra &\lesssim 2^{-k} \|\la
x\ra^{-2} u\|^2_{L^2_{t,x}} 
\label{com.high}
\end{align}
for all functions $u$ frequency localized at frequency $2^k$.  Here, $\rho_{<k}=S_{<k-4}\rho$ where
$\rho$ is as in the definition of $\tX_e$.
\end{lemma}

\begin{proof}
We replace the sequence $(\alpha_m)$ by a larger one satisfying
an analogue of \eqref{new.alpha}, namely 
\begin{equation}
\label{new.alpha.high}
\begin{cases}
(\alpha_m)\text{ is slowly varying,}\\
\alpha_0 = 1 \\
\displaystyle \sum_{m\geq 0} \alpha_m \approx 1,\\
\kappa_m \le \epsilon \alpha_{m}\text{ for }  m \geq 0,
\end{cases}
\end{equation}
and let $\alpha$ be a slowly varying function satisfying
\[
\alpha(s) \approx \alpha_m, \qquad s \approx 2^m.
\]
We  construct $\phi$ as in the low frequency case
so that  \eqref{phi1} and \eqref{phi2} are satisfied.
Then we set
\[
Q_k=2^{-k} \delta(D_ia^{ij}_{(k)}x_j\phi(\delta|x|)+
\phi(\delta|x|)a^{ij}_{(k)}x_iD_j).
\]
This choice is not very different from the one in the low frequency
case. The metric $a^{ij}$ is inserted in order to insure a crucial
sign condition in the proof \eqref{com.high}.

The first property, \eqref{Q.L2.high}, is immediate from the
properties of $\phi$ and \eqref{coeffak.greater}.  The bound
\eqref{Q.X.high} is also straightforward since the coefficients
 $a^{ij}_{(k)}$ are bounded.

 {\bf Proof of \eqref{C.high}:} In order to prove \eqref{C.high}, we
 calculate (using the symmetry of $a^{ij}$)
\begin{equation}
\label{C.high.calc} \begin{split}
C_k=&\delta 2^{-k}\Bigl[2D_l a^{lm}_{(k)}\partial_m(a^{ij}_{(k)}x_j \phi(\delta|x|))D_i
+2D_i\partial_l(a^{ij}_{(k)}x_j\phi(\delta |x|))a^{lm}_{(k)}D_m
\\&-2D_l\partial_i(a^{lm}_{(k)})a^{ij}_{(k)}x_j\phi(\delta |x|)D_m - \partial_l(a^{lm}_{(k)}
\partial_i\partial_m(a^{ij}_{(k)} x_j \phi(\delta |x|)))
\Bigr].
\end{split}\end{equation}
The main positive contribution is obtained by substituting
$a$ by $I_n$ in the first two terms,
\[
\begin{split}
C_k^0 = &\ 2^{-k}\delta \left[2D_l \partial_l (x_i \phi(\delta|x|))D_i
+2D_i\partial_l(x_i\phi(\delta |x|)) D_l\right]
\\ =  &\ 4\cdot  2^{-k}\delta \left[ D \phi(\delta |x|)D +  D
  \frac{x}{|x|} \delta |x|\phi'(\delta |x|)\frac{x}{|x|}D \right].
\end{split}
\]
As in the low frequency case, this satisfies
\[
\langle C_k^0 u,u\rangle \gtrsim \delta 2^{k} \left\langle
\frac{\alpha(\delta |x|)}{\la \delta |x|\ra}u,u \right\ra
\]
for any function $u$ localized at frequency $2^k$.  The other
contributions are shown to be smaller error terms. Consider for
instance the error made by substituting $a^{ij}_{(k)}$ by $I_n$ in the first
term. By  \eqref{coeffak.greater}, we  can estimate
\[
| a^{lm}_{(k)}\partial_m (a^{ij}_{(k)} x_j \phi(\delta |x|)) -
 \delta^{lm} \partial_m (\delta^{ij} x_j \phi(\delta |x|)) |\lesssim
\frac{\kappa_k(|x|)}{\la \delta x\ra}
\]
which contributes to $\la C_k u,u \ra$ an error of size
\[
 \delta 2^{k} \left\langle
\frac{\kappa_k (|x|)}{\la \delta x\ra}u,u \right\ra.
\]
A similar contribution comes from the second term and the third 
term. Finally, for the last term in $C$ we have
\[
|\partial_l(a^{lm}_{(k)}\partial_i\partial_m(a^{ij}_{(k)}x_j\phi_k(\delta|x|)))|
\lesssim
\frac{2^k \kappa_k(|x|)}{\la x \ra \la \delta x\ra} +
\frac{\delta^2}{\la \delta x\ra^3} \lesssim \frac{2^k \kappa_k(|x|)}
{ \la \delta x\ra} +
\frac{\delta^2\alpha(\delta |x|)}{\la \delta x\ra}
\]
which yields an error of size
\[
 \delta \left\langle
\frac{\kappa (|x|)}{\la \delta x\ra}u,u \right\ra + \delta^3 2^{-k} \left\langle
\frac{\alpha (\delta|x|)}{\la \delta x\ra}u,u \right\ra.
\]
Summing up we have proved that
\begin{equation}
\langle C_k u,u\rangle \geq c_1 \delta 2^{k}\Bigl\la 
\frac{\alpha( \delta |x|)}{\la \delta x\ra}u,u\Bigr\ra - 
c_2 \delta^3 2^{-k}\Bigl\la \frac{\alpha(\delta |x|)}{\la \delta x\ra}
u,u\Bigr\ra - 
c_3 \delta 2^{k} \left\langle \frac{\kappa(|x|)}{\la  \delta x\ra}u,u\right\ra.
\end{equation}
Choosing $\delta$ small enough (independently of $(\alpha_m)$ and $k$),
the second term on the right is negligible compared to the first.
Since $\alpha$ is slowly varying, by \eqref{new.alpha.high} the last
term is also negligible provided that $\epsilon$ is sufficiently
small. Hence \eqref{C.high} follows.

{\bf Proof of \eqref{com.high}:} We 
denote by $L$ the self-adjoint operator
\[
L = x_i a^{ij}_{(k)} D_j + D_i a^{ij}_{(k)} x_j
\]
and begin by calculating
\[
\begin{split}
\frac{1}i  [A_{(k)},\rho_{<k}]&\ =-D_ia^{ij}_{(k)}(\partial_j\rho_{<k})
 - a^{ij}_{(k)} (\partial_i \rho_{<k})D_j.
\\ &\ = -  |x|^{-1} \rho'_{<k}  L + i x_i a^{ij}_{(k)} \partial_j
(|x|^{-1}   \rho'_{<k})
\end{split}
\]
and
\[
2^k Q_k \rho_{<k} =   \delta \rho_{<k}  \phi(\delta |x|) L   -  i  x_i   a^{ij}_{(k)}
(\rho_{<k} \partial_j \phi(\delta |x|) + 2 \phi(\delta |x|) \partial_j\rho_{<k}).
\] 
Thus, after one integration by parts we obtain
\begin{multline}\label{com.expanded}
2^k \Im\la [A_{(k)}, \rho_{<k}]u,Q\rho_{<k} u\ra 
= -\delta \int |x|^{-1}  {\rho'_{<k}} \phi(\delta |x|)
\rho_{<k}  |L u|^2 dx dt
+ \int V |u|^2 dx dt
\end{multline}
where the scalar function $V$ is given by 
\[
\begin{split}
V =&\ (x_i a^{ij}_{(k)} \partial_j + \partial_i a^{ij}_{(k)}x_j) \left( \rho_{<k}  \phi(\delta |x|) 
 x_l a^{lm}_{(k)} \partial_m
(|x|^{-1}   \rho'_{<k}) \right)
\\ &\ +
  (x_i a^{ij}_{(k)} \partial_j+\partial_i a^{ij}_{(k)} x_j)  \left(  |x|^{-1} \rho'_{<k}    x_l   a^{lm}_{(k)}
(\rho_{<k} \partial_m \phi(\delta |x|) + 2 \phi(\delta |x|) \partial_m\rho_{<k})  \right)
\\ &\ - \left(
 x_i a^{ij}_{(k)} \partial_j
(|x|^{-1}   \rho'_{<k}) \right)
\left(     x_l   a^{lm}_{(k)}
[\rho_{<k} \partial_m \phi(\delta |x|) + 2 \phi(\delta |x|) \partial_m\rho_{<k}]\right). 
\end{split}
\]
Morally speaking, the first term in \eqref{com.expanded} is negative
and can be dropped.  This is true modulo the tails that are introduced
by the frequency cutoff which is applied to $\rho$.  Since
\[
|r^{-1} (\rho'(r) - \rho'_{<k}(r))| \lesssim 2^{-Nk} \la r\ra^{-N},
\]
 the error is estimated by
\[
 2^{-Nk} \|\la x \ra^{-2} u\|^2_{L^2_{t,x}}.
\]
On the other hand the weight $V$ is bounded and rapidly decreasing at
infinity,
\[
|V| \lesssim \la x\ra^{-N},
\]
from which \eqref{com.high} follows.
\end{proof}

We now return to the proof of Proposition~\ref{Qprop}.  We choose $Q_\rho $
of the form
\[
Q_\rho  = \sum_{k=-\infty}^\infty \rho S_k Q_k S_k   \rho
\]
where for each $k$ we have an $L^2$ bounded self-adjoint operator 
localized at frequency $2^k$. 

The $L^2$ boundedness of $Q_\rho $ follows from the $L^2$ boundedness
of $Q_k$, and the $\tX$ boundedness of $Q_\rho $ follows from the 
$X_{k,\alpha}$ boundedness of $Q_k$ after optimizing in $\alpha$. 
It remains to consider the commutator 
$C$. We write
\[
C = i \sum_k [A, \rho S_k Q_k S_k   \rho ].
\]
We first replace $A$ by $A_{(k)}$ and $\rho$ by $\rho_{<k}$
for $k > 0$ and by $1$ for $k < 0$. This generates error terms
which we need to estimate.

If $k < 0$ then these error terms are estimated
as follows.  We first want to substitute $A$ by $A_{(k)}$, and as such, we see errors of the
form
\begin{equation}\label{error1}
 |\la [A,\rho]u,\sum_{k<0} S_k Q_k S_k \rho u\ra|
+ |\sum_{k<0} \la (A-A_{(k)})\rho u, S_k Q_k S_k \rho u\ra|.
\end{equation}
For the first term, we use \eqref{Q.X.low} (after optimizing in $\alpha$)
\[
\begin{split}
|\la [A,\rho]u, \sum_{k<0}S_k Q_k S_k \rho u\ra|&\lesssim
|\la -2iD_i a^{ij}(\partial_j \rho) u + \partial_j((\partial_i \rho)a^{ij})u,\sum_{k<0}S_k Q_k S_k \rho u\ra|\\
&\lesssim \|u\|_{L^2_{t,x}(2^{M}< |x|< 2^{M+1})} \|\sum_{k<0}S_k Q_k S_k \rho u\|_{L^2_{t,x}(\{
2^{M}<|x|<2^{M+1}\})}\\
&\lesssim \|\la x\ra^{-2} u\|_{L^2_{t,x}} \|\sum_{k<0} S_k Q_k S_k \rho u\|_{\dX}\\
&\lesssim \|\la x\ra^{-2} u\|_{L^2_{t,x}} \|\rho u\|_{\tX}.
\end{split}
\]
For the second term in \eqref{error1}, we use \eqref{aak} and \eqref{Q.X.low} to see that
\[
\begin{split}
  |\sum_{k<0} \la (A-A_{(k)})\rho u, S_k Q_k S_k \rho u\ra|&\lesssim
\Bigl(\sum 2^{-k} \|S_k(A-A_{(k)})\rho u\|^2_{X'_k}\Bigr)^{\frac12} \|\rho u\|_{\dX}\\
&\lesssim \e \|\rho u\|^2_{\tX}. 
\end{split}
\]
For the remaining errors, we use the fact that $A_{(k)}$ preserves
localizations at frequency $2^k$ combined with
\eqref{coeffak.greater}, and \eqref{Q.L2.low} to see that
\[
\begin{split}
  |\sum_{k<0} \la A_{(k)} (1-\rho) u, S_k Q_k S_k \rho u\ra|
&\lesssim \|\la x\ra^{-2} u\|_{L^2_{t,x}} \|\la x\ra^2 \sum_{k<0} S_k Q_k S_k A_{(k)} (1-\rho)u\|_{L^2_{t,x}}\\
&\lesssim \|\la x\ra^{-2} u\|_{L^2_{t,x}} \|(1-\rho)u\|_{L^2_{t,x}}
\end{split}
\]
and respectively,
\[
\begin{split}
  |\sum_{k<0} \la A_{(k)} u, S_k Q_k S_k (1-\rho)u\ra|&\lesssim
\|\la x\ra^{-2} u\|_{L^2_{t,x}}\|\la x\ra^2 A_{(k)} \sum_{k<0}S_k Q_k S_k (1-\rho)u\|_{L^2_{t,x}}\\
&\lesssim \|\la x\ra^{-2} u\|_{L^2_{t,x}} \|(1-\rho)u\|_{L^2_{t,x}}.
\end{split}
\]
In both formulas above the last step is achieved by commuting the
$x^2$ factor to the right, where it is absorbed by the $(1-\rho)$
factor.  The two possible commutators may yield an extra $2^{-2k}$
factor, which is compensated for by the two derivatives in $A_{(k)}$.

On the other hand if $k \geq 0$ then we have the bound
\[
| \rho - \rho_{<k}| \lesssim 2^{-Nk} \la x \ra^{-N}.
\]
This estimate clearly provides summability in $k$, and 
the control for the correction terms similar to the
above ones follows from analogous arguments.  The terms, e.g., of the form
$\|(1-\rho)u\|_{L^2_{t,x}}$ are simply replaced by 
$\|\la x\ra^{-2} u\|_{L^2_{t,x}}$.

Hence we are left with the modified commutator
\[
\tilde C =  i \sum_{k < 0}  [A_{(k)}, S_k Q_k S_k  ]+ i
\sum_{k \geq 0}  [A_{(k)}, \rho_{<k} S_k Q_k S_k   \rho_{<k} ]
\]
where all terms are now frequency localized.  The first term is rewritten
in the form
\[
 i[A_{(k)}, S_k Q_k S_k  ] =  i[A_{(k)}, S_k] Q_k S_k + iS_k Q_k [A_{(k)},
 S_k] + S_k C_k S_k.
\]
For the first two terms we use the commutator estimate \eqref{comak} and the
$X_k$ boundedness of $Q_k$ \eqref{Q.X.low}.  We can, thus,
bound the corresponding inner products by 
\[
\epsilon \|\rho u\|^2_{\tX} + \e \|(1-\rho) u\|^2_{L^2_{t,x}}. 
\]
For the third term, we shall use \eqref{C.low}.

Next we consider the high frequency terms in $C$,
\begin{multline*}
 [A_{(k)}, \rho_{<k} S_k Q_k S_k   \rho_{<k} ] = 
  \rho_{<k} [A_{(k)}, S_k Q_k S_k]   \rho_{<k}  +  
 [A_{(k)}, \rho_{<k}] S_k Q_k S_k   \rho_{<k} \\+
 \rho_{<k} S_k Q_k S_k [A_{(k)}, \rho_{<k}].
\end{multline*}
The first term is treated as above but using \eqref{C.high} instead.
For the remaining two terms we commute both outside factors inside.
This yields a main contribution which is estimated by \eqref{com.high},
\[
\begin{split}
2 \Im \langle  [A_{(k)}, \rho_{<k}] S_k u,  Q_k \rho_{<k} S_k u \rangle
&\ \lesssim 2^{-k} \| \langle x\rangle^{-2} S_k u\|_{L^2_{t,x}}^2. 
\end{split}
\]
The remaining terms involve an extra commutation which kills the
remaining derivative in $A_{(k)}$. Also $ \rho_{<k}$ is differentiated,
which yields rapid decay at infinity. Hence we can bound them by 
\[
\| \la x \ra^{-2} S_k u\|_{L^2_{t,x}}^2.
\]

Summing up, we have proved that
\[
\begin{split}
\langle C u, u \rangle \geq &\ c_1\left( \sum_{k < 0} 2^k \|   S_k
  u\|_{X_{k.\alpha(k)}}^2 + \sum_{k > 0} 2^k \|   S_k \rho_{<k}
  u\|_{X_{k,\alpha(k)}}^2 \right)
\\ &\
- c_2 \left( \|\la x \ra^{-2}
u\|_{L^2_{t,x}}^2+ 
\e\|\rho u\|_{\tX}^2 \right)
\end{split}
\]
where for each $k$ we have used a different $\alpha$ denoted by
$\alpha(k)$. Optimizing with respect to all choices of $\alpha(k)$ we
obtain
\[
\begin{split}
\langle C u, u \rangle \geq &\ c_1\left( \sum_{k < 0} 2^k \|   S_k
  u\|_{X_{k}}^2 + \sum_{k > 0} 2^k \|   S_k \rho_{ < k}
  u\|_{X_{k}}^2 \right)
\\ &\
- c_2 \left( \|\la x \ra^{-2}u\|_{L^2_{t,x}}^2 + 
\e
\|\rho u\|_{\tX}^2 \right)
\end{split}
\]
which for $\e$ sufficiently small yields part (iii) of the proposition. 
\end{proof}

%%%%%%%%%%%%%%%%%%%%%%%%%%%%%%%%%%%%%%%%%%%%%%%%%%
\subsection{The non-resonant low dimensional case $n=1,2$:  Proof of
  Theorem~\ref{main.ls.theorem}}

 Almost all the arguments
in the high dimensional case apply also in low dimension. The only
difference arises in part (ii) of Proposition \ref{Qprop}. Since
the multiplication by $\rho$ is bounded in both $\tX$ and $\tX'$,
the property (ii) reduces to proving that
\[
\sum_{k=-\infty}^\infty S_k Q_k S_k : \tX \to \tX.
\]
In dimension $n \geq 3$ the $\tX$ norm is described in terms of the
$X_k$ norms of its dyadic pieces, and the above property follows
from the $X_k$ boundedness of $Q_k$ at frequency $2^k$. 

However, in dimension $n=1,2$ the $\tX$ norm also has a weighted $L^2$
component. The high frequency part $k \geq 0$ of the above sum causes
no difficulty, but the low frequency part does. We do know that
\[
\sum_{k=-\infty}^0 S_k Q_k S_k : \dX \to \dX.
\]
Therefore, due to Lemma~\ref{txdx}, it would remain to prove that
\[
\Bigl\|  \sum_{k=-\infty}^0 S_k Q_k S_k u\Bigr\|_{L^2_{t,x}(\{|x| \leq 1\})} \lesssim
\|u\|_{\tX}.
\]
Unfortunately, the operators $S_k Q_k S_k$ act on the $2^{-k}$ spatial
scale; therefore without any additional cancellation there is no
reason to expect a good control of the output in a bounded region. The
aim of the next few paragraphs is to replace the above low frequency
sum by a closely related expression which exhibits the desired
cancellation property.

First of all, it is convenient to replace the discrete parameter $k$
by a continuous one $\sigma$. The operators $S_\sigma$ are defined in
the same way as $S_k$ by scaling. Let $\phi_k$ be the functions in
Lemma~\ref{lemma.low.freq}.  The functions $\phi_\sigma$ are defined
from $\phi_k$ using a partition of unity on the unit scale in $\sigma$.
The normalization we need is very simple, namely $\phi_k(0) = 1$,
which leads to $\phi_{\sigma}(0) = 1$. The operators $Q_\sigma$ are
defined in a similar way. Then it is natural to substitute
\[
 \sum_{k=-\infty}^0 S_k Q_k S_k  \to \int_{-\infty}^0  S_\sigma
 Q_\sigma S_\sigma d\sigma
\]
and all the estimates for the second sum carry over identically
from the discrete sum.

However, the desired cancellation is still not present in the second
sum. To obtain that we consider a spherically symmetric Schwartz function 
$\phi^0$ localized at frequency $\ll 1$ with $\phi^0(0) = 1$. 
Then we write $\phi_\sigma$ in the form
\[
\phi_\sigma(x) = \phi^0(x) + x^2 \psi_\sigma(x).
\]

The modified self-adjoint operators $\tQ_\sigma$ are defined as 
\[
\tQ_\sigma =  S_\sigma Q_{\sigma,\phi^0}  S_\sigma + 2^{2\sigma} \delta^2 x S_\sigma
Q_{\sigma,\psi_\sigma}  S_\sigma x
\]
where, as in Lemma~\ref{lemma.low.freq}, we set
\[
Q_{\sigma,\phi} =  \delta(Dx\phi( 2^\sigma \delta |x|)+\phi(2^\sigma \delta |x|)xD).
\]

We claim that the conclusion of Proposition~\ref{Qprop} is valid with
the operator $Q$ defined as 
\begin{equation}
Q_\rho  = \rho Q \rho, \qquad Q = \int_{-\infty}^0 \tQ_\sigma d\sigma +  \sum_{k=0}^\infty S_k Q_k S_k. 
\label{qlowd}
\end{equation}
The family $\mathcal Q$ is obtained as before by allowing the
choice of the functions $\phi_k$ to depend on the slowly varying
sequences $(\alpha^\sigma_{j})_{j \in \N}$ which are chosen
independently\footnote{In effect, without any restriction in generality,
  one may also assume that $\alpha^\sigma_{j}$ is also slowly varying
  with respect to $\sigma$} for different $k$.

There is no change in part (i) of Proposition~\ref{Qprop}. For part
(ii) we need to prove that
\begin{equation} \label{qtx}
\left\| Q  u \right \|_{\tX} \lesssim \|u\|_{\tX}.
\end{equation}
The high frequencies are estimated directly from the $\dX$ norm;
therefore we have to consider the integral term in $Q$ and show that
\[
\left\| \int_{-\infty}^0  \tQ_\sigma  u\, d\sigma\right \|_{\tX} \lesssim \|u\|_{\tX}.
\]

The $\dX$ component of the $\tX$ norm is easily estimated by Littlewood-Paley
theory, so due to Lemma~\ref{txdx}, it would remain to prove the local
$L^2$ bound
\begin{equation}\label{quo}
\left\| \int_{-\infty}^0  \tQ_\sigma  u d\sigma\right \|_{L^2_{t,x}(\{|x| \leq 1\})} \lesssim
\|u\|_{\tX}.
\end{equation}

We can neglect the time variable in the sequel.
We have the $L^2$ bound
\[
\| \tQ_{\sigma} u\|_{X_\sigma} \lesssim
 \|  S_\sigma u\|_{X_\sigma}
\]
which leads to 
\[
\| \nabla \tQ_{\sigma} u\|_{X_\sigma} \lesssim
2^{\sigma} \| S_\sigma u\|_{X_\sigma}
\]
and the corresponding pointwise bound
\[
\|  \nabla \tQ_{\sigma} u\|_{L^\infty (A_{<-\sigma})} \lesssim
2^{\frac{n+1}2 \sigma } \|  S_\sigma u\|_{X_\sigma}
\]
which establishes the convergence and the bound for the corresponding integral
\[
\left\| \int_{-\infty}^0   \nabla \tQ_{\sigma} u\,d\sigma \right\|_{L^\infty(A_{<0})} \lesssim
\| S_{\leq 0} u\|_{\dX}.
\]
Hence in order to prove \eqref{quo} it remains to establish a similar
bound for the integral at $x=0$. Assume first that $u \in L^2$, which
arguing as above guarantees the uniform convergence of the integral.
Denoting by $K_\sigma$ the kernel of $S_\sigma$ we have
\[
\begin{split}
(\tQ_{\sigma} u)(0) = &\  (S_\sigma Q_{\sigma,\phi^0}  S_\sigma) u(0) 
\\
= &\ \la K_\sigma, Q_{\sigma,\phi^0}  S_\sigma u \ra
=  \la Q_{\sigma,\phi^0} K_\sigma, S_\sigma u \ra 
\\ 
= & \ \int Q_{\sigma,\phi^0}(x,D_x) K_\sigma(x)  \int  K_\sigma(x-y) u(y) dy
dx
\\ =&\ (S_\sigma^1 u)(0)
\end{split}
\]
where $S_\sigma^1$ is the frequency localized multiplier with
spherically symmetric Schwartz kernel
\[
K_\sigma^1 = Q_{\sigma,\phi^0}(x,D_x) K_\sigma * K_\sigma.
\]
Due to the frequency localization we can define
\[
S_{<0}^1 = \int_{-\infty}^{0} S_\sigma^1 d \sigma.
\]
The punch line is that by construction the operators $S_\sigma^1$ have
the same kernel up to the appropriate rescaling. This implies that the
symbols of $S_{<0}^1$ are constant for $|\xi| \leq
2^{-4}$.  Hence both the symbols and the kernels
$K_{<0}^1$ of $S_{<0}^1$ are Schwartz functions which
coincide modulo rescaling. Hence for all functions $u \in L^2$ we have
\[
\int_{-\infty}^0  \tQ_{\sigma} u (0) d\sigma = \la K_{<0}^1, u\ra
\]
which leads to the estimate
\[
\left | \int_{-\infty}^0  (\tQ_\sigma  u)(0) d\sigma\right| \lesssim
\|u\|_{\tX}.
\]
This completes the proof of the estimate \eqref{quo} for all $u \in
L^2$, and, by density, shows that the integral
\[
\int_{-\infty}^0   \tQ_\sigma d\sigma
\]
has a unique bounded extension to $\tX$.

It remains to prove part (iii) of Proposition~\ref{Qprop}. If
$\tQ_\sigma$ is replaced by $S_\sigma Q_\sigma S_\sigma$ then the high
dimensional argument applies by simply replacing sums with integrals.
Hence it remains to estimate the difference. Commuting we obtain
\[
\tQ_\sigma - S_\sigma Q_\sigma S_\sigma =  i \delta^2 2^{2\sigma} \left( S'_\sigma
Q_{\sigma,\psi} x S_\sigma - S_\sigma x Q_{\sigma,\psi}   
S'_\sigma  -   S'_\sigma Q_{\sigma,\psi} S'_\sigma (D) \right).
\]
Commuting again to take advantage of the cancellation between the
first two terms, by semiclassical pdo calculus we can write
\[
\tQ_\sigma - S_\sigma Q_\sigma S_\sigma = \delta^2 R_\sigma (2^\sigma
\delta x, 2^{-\sigma} D)
\]
where the symbol $r_\sigma(y,\eta)$  is localized in $\{ |\eta|
\approx 1\}$ and satisfies
\[
|\partial_y^\alpha \partial_\eta^\beta r_\sigma(y,\eta)| \leq
c_{\alpha \beta} \la y \ra^{-2}.
\]
This implies the bound
\[
\| (\tQ_\sigma - S_\sigma Q_\sigma S_\sigma) u \|_{X'_\sigma} \lesssim
\delta^2 2^{-\sigma} \|S_\sigma u\|_{X_\sigma}. 
\]
Therefore without any commuting we obtain
\[
| \la [\tQ_\sigma - S_\sigma Q_\sigma S_\sigma,A_{(\sigma)}] u,u\ra |
\lesssim \delta^2 \| u \|_{\dX}^2.
\]
This error is negligible since, as one can note in the proofs of
Lemmas~\ref{lemma.low.freq}, \ref{lemma.high.freq},
 the constant $c_1$ in (iii)  has size $c_1 = O(\delta)$.

 \subsection{The resonant low dimensional case $n=1,2$: Proof of
   \ref{main.ls.theorem.res}}

The proof follows the same outline as in the non-resonant case, with minor
modifications. The energy estimate \eqref{eest} is now replaced
by 
\begin{equation}
\| u\|_{L^\infty_t L^2_x}^2 \lesssim \| u_0\|_{L^2}^2 + \| f_1\|_{L^1_t L^2_x}^2
+ \| u\|_{\dX_e} \|f_2\|_{\dX'_e}. 
\label{eest.res}\end{equation}
Instead of the exterior smoothing estimate \eqref{lsest},
we need to prove
\begin{equation}
  \| T_\rho u\|_{\dX}^2 \lesssim \| u\|_{L^\infty_t L^2_x}^2 + 
  \| f_1\|_{L^1_t L^2_x}^2 
  +  \|T_\rho f_2\|_{\dX'}^2 + \| \la x \ra^{-2}  (u-u_\rho) \|_{L^2_{t,x}}^2.
\label{lsest.res}\end{equation}
The estimate \eqref{lsext.res} then follows from the previous two estimates as
well as \eqref{dXerror}.

The lower order terms will still be negligible.  Indeed, letting
$B=2b^i D_i$, we have
\[
 T_\rho Bu =  B T_\rho u - (B \rho)(u-u_\rho)   + (1-\rho) \left( \int
   (1-\rho) dx\right)^{-1}  \int (B\rho)(u-u_\rho) dx.
\]
Therefore by \eqref{bnoc}, we obtain
\[
\| T_\rho Bu\|_{\dX'} \lesssim \e \|u\|_{\dX_e},
\]
which combined with the $\dX$ boundedness of our multiplier below shows that the lower order
terms can be neglected.

The estimate \eqref{lsest.res} follows from 

\begin{proposition} \label{Qprop.res}
There is a family $\cal Q_{res}$ of bounded self-adjoint operators $Q_{res}$ with 
the following properties:

(i) $L^2$ boundedness,
\[
\|Q_{res}\|_{L^2 \to L^2} \lesssim  1,
\]

(ii) $\dX$ boundedness,
\[
|\langle Q_{res} u, f \rangle | \lesssim \| T_\rho f\|_{\dX'} \| T_\rho u \|_{\dX},
\]

(iii) Positive commutator,
\[
\sup_{Q_{res} \in \cal Q_{res}} \langle Cu,u \rangle \geq  c_1 \|T_\rho u\|_{\dX}^2 -
c_2 \|\la x\ra^{-2}  (u-u_\rho) \|_{L^2_{t,x}}^2.
\]
\end{proposition}

\begin{proof}
  We construct $Q_{res}$ as in the non-resonant case  but with the
  modified truncation operator
\[
Q_{res} u =  T_\rho Q T_\rho.
\]
with $Q$ given by \eqref{qlowd}.

The properties (i) and (ii) are straightforward.
For (iii) we note that 
\[
S_k T_\rho u = S_k \rho(u-u_\rho)
\]
while
\[
 T_\rho  A u =  \rho Au + c (1-\rho) \int (1-\rho) A (u-u_\rho) dx = 
 \rho A(u-u_\rho) - c (1-\rho) \int  (u-u_\rho) A \rho dx.
\]
Hence we can express the bilinear form $\la Au, Q_{res} u \ra$ in
terms of the operator $Q_\rho$ in the nonresonant case
\[
\la Au, Q_{res} u \ra = \la A(u-u_\rho), Q_\rho (u-u_\rho) \ra - c \int
(u-u_\rho) A \rho dx \  \la (1-\rho), Q T_\rho u \ra
\]
which implies that
\[
\la C_{res} u,u\ra = \la C (u-u_\rho),u-u_\rho\ra + c \Im \int
(u-u_\rho) A \rho dx\  \la (1-\rho), Q T_\rho u \ra.
\]
Hence we  can apply part (iii) of Proposition~\ref{Qprop} and
\eqref{qtx} to obtain the desired conclusion.
\end{proof}

%%%%%%%%%%%%%%%%%%%%%%%%%%%%%%%%%%%%%%%%%%%%%%%%%%
\subsection{Non-trapping metrics: Proof of Theorem~\ref{main.ls.theoremnt}.}
This requires some modifications of the previous argument. First of
all, instead of the energy estimate \eqref{eest}, we need a
straightforward modification of it, namely
\begin{equation}
\| u\|_{L^\infty_t L^2_x}^2 \lesssim \| u_0\|_{L^2}^2 + \| f_1\|_{L^1_t L^2_x}^2
+ \| u\|_{\tX} \|f_2\|_{\tX'}.
\label{eestnt}\end{equation}
We still need the exterior local smoothing estimate \eqref{lsest}.
However, now we can complement it with an interior estimate,
namely
\begin{equation}
\| (1-\rho) u\|_{\tX}^2 \lesssim \| u\|_{L^\infty_t L^2_x}^2 + 
\| f_1\|_{L^1_t L^2_x}^2  + \| \rho u\|_{\tX}^2
+  \|(1-\rho) f_2\|_{\tX'}^2 + \| (1-\rho) u\|_{L^2_{t,x}}^2. 
\label{lsestnt}\end{equation}
The conclusion of Theorem~\ref{main.ls.theoremnt} is obtained by combining 
the three estimates \eqref{eestnt}, \eqref{lsest} and \eqref{lsestnt}.

It remains to prove \eqref{lsestnt}. This is obtained by applying 
to the function $v =  (1-\rho) u$ the local bound

\begin{proposition}
  Assume that the coefficients $a^{ij}$, $b^i$, $c$ are real and satisfy \eqref{coeff}, \eqref{coeffb},
  and \eqref{coeffc}. Moreover, assume that
  the metric $a^{ij}$ is non-trapping. Let $v$ be a
  function supported in $\{|x| \leq 2^{M+1}\}$ which solves the
equation
\begin{equation}
(D_t +A) v = g_1 + g_2, \qquad v(0) = v_0
\end{equation}
in the time interval $[0,T]$. 
Then we have 
\begin{equation}
\| v \|_{L^2_t H^{\frac12}_x}^2 \lesssim \| v\|_{L^\infty_t L^2_x}^2 + 
\| g_1\|_{L^1_t L^2_x}^2  +  \|g_2\|_{L^2_t H^{-\frac12}_x}^2 + \| v\|_{L^2_{t,x}}^2. 
\label{leint}\end{equation}
\end{proposition}

\begin{proof}
  We use again the multiplier method.  The following lemma tells us
  how to choose an appropriate multiplier.
\begin{proposition}\label{Doi_construction}
  Assume that the coefficients $a^{ij}$ satisfy \eqref{coeff}.
  Moreover, we assume that the Hamiltonian vector field $H_a$ permits
  no trapped geodesics.  Then there exists a smooth, time-independent, real-valued symbol $q\in S^0_{hom}$
  so that
\[
H_a q\gtrsim |\xi|,\quad \text{ in } \{|x| \leq 2^{M+1}\}.
\]
\end{proposition}

This proposition is essentially from \cite{D1}, if $a^{ij}$ were smooth.  See also 
Lemma 1 of \cite{ST}, which includes some discussion of the limited regularity. 

Working in the Weyl calculus and using this multiplier $Q$,  we compute
\[
\frac{d}{dt}\la v,Qv\ra = -2\Im\la (D_t+A)v,Qv\ra
+ i\la [A,Q]v,v\ra 
\]
which after time integration yields
\[
\la i [A,Q]v,v\ra = \la v,Qv \ra |^T_0 + 2\Im\la g_1+g_2,Qv\ra. 
\]
For the second term on the right,  we  apply Cauchy-Schwarz and use the
$L^2$ and $H^\frac12$ boundedness of $Q$ to obtain
\[
|\la (D_t+A)v,Qv\ra| \lesssim \|v\|_{L^\infty_t L^2_x}^2  + \|g_1\|_{L^1_t
  L^2_x}^2  + \|g_2\|_{L^2_t H_x^{-\frac12}} \| v\|_{L^2_t H_x^\frac12}.
\]
Hence
\[
\la i [A,Q]v,v\ra \lesssim 
\|v\|_{L^\infty_t L^2_x}^2  + \|g_1\|_{L^1_t
  L^2_x}^2  + \|g_2\|_{L^2_t H^{-\frac12}_x} \| v\|_{L^2_t H^\frac12_x}.
\]
Then it remains to prove the positive commutator bound
\begin{equation}
\la i [A,Q]v,v\ra \geq c_1 \|v\|_{L^2_t H^\frac12_x}^2 - c_2 \|v\|_{L^2_{t,x}}^2.
\end{equation}
The positive contribution comes from the second order terms in $P$.
Precisely, we have
\[
i[ D_i a^{ij} D_j,Q(x,D)] = Op(H_a q) +O(1)_{L^2 \to L^2}.
\]
The first symbol is positive, and we can obtain a bound from below by
G\aa rding's inequality.  The first order term yields an $L^2$ bounded commutator, and the zero
order term is $L^2$ bounded by itself.
 
Here, we remind the reader that
we are not working with classical smooth symbols but instead with
symbols of limited regularity, and we refer the interested reader to
the discussion in Taylor \cite[p. 45]{TaylorIII} for further details
on these otherwise classical results.
\end{proof}

%%%%%%%%%%%%%%%%%%%%%%%%%%%%%%%%%%%%%%%%%%%%%%%%%%
\subsection{Non-trapping metrics: Proof of Theorem~\ref{main.ls.theoremnt.res}.}

The argument is similar to the above one, with some obvious
modifications. Instead of \eqref{eestnt} we have
\begin{equation}
\| u\|_{L^\infty_t L^2_x}^2 \lesssim \| u_0\|_{L^2}^2 + \| f_1\|_{L^1_t L^2_x}^2
+ \| u\|_{\dX} \|f_2\|_{\dX'} 
\label{eestnt.res}\end{equation}
while \eqref{lsestnt} is replaced by 
\begin{multline}
\| (1-\rho) (u-u_\rho)\|_{\dX}^2 \lesssim \| u\|_{L^\infty_t L^2_x}^2 + 
\| f_1\|_{L^1_t L^2_x}^2  + \| \rho (u-u_\rho)\|_{\dX}^2
+  \|f_2\|_{\dX'}^2 \\+ \| \la x\ra^{-2} (u-u_\rho)\|_{L^2_{t,x}}^2. 
\label{lsestnt.res}\end{multline}

The conclusion of Theorem~\ref{main.ls.theoremnt.res} is obtained by combining 
the estimates \eqref{eestnt.res}, \eqref{lsest.res} and \eqref{lsestnt.res} and
applying \eqref{dXerror} to reduce the error terms to the form presented in \eqref{lsntres}.

It remains to prove \eqref{lsestnt.res}.  We first compute
\[\begin{split}
D_t u_\rho &= \Bigl(\int (1-\rho)\:dx\Bigr)^{-1}\Bigl[\la (D_t+A)u, (1-\rho)\ra
- \la Au, (1-\rho)\ra\Bigr]\\
&= \Bigl(\int(1-\rho)\:dx\Bigr)^{-1}\Bigl[\la f_1+f_2,(1-\rho)\ra - \la u-u_\rho,A(1-\rho)\ra\Bigr].
\end{split}
\]
The function $v = (1-\rho) (u-u_\rho)$ solves
\begin{multline*}
P v = (1-\rho)(f_1+f_2) -(1-\rho)\Bigl(\int(1-\rho)\:dx\Bigr)^{-1}\Bigl[\la f_1+f_2,(1-\rho)\ra
-\la u-u_\rho, A(1-\rho)\ra\Bigr]
\\+[A,(1-\rho)](u-u_\rho).
\end{multline*}
Then we apply \eqref{leint} to $v$ to obtain
\[
\begin{split}
  \| v\|_{L^2_t H^\frac12_x}^2 &\lesssim  \|v\|_{L^\infty_t L^2_x}^2 + \|
  (1-\rho) f_1 \|_{L^1_t L^2_x}^2 +\| [A,(1-\rho)] (u-u_\rho)\|_{L^2_t
    H^{-\frac12}_x} \\ &\qquad\qquad\qquad\qquad\qquad\qquad\qquad + \| (1-\rho) f_2 \|_{L^2_t H^{-\frac12}_x}^2 
   + \| \la x\ra^{-2}(u- u_\rho)\|_{L^2_{t,x}}^2 \\ & \
  \lesssim \|u\|_{L^\infty_t L^2_x}^2 + \| f_1 \|_{L^1_t L^2_x}^2 +
  \|\rho(u-u_\rho)\|_{\dX}^2 + \| f_2 \|_{\dX'}^2 
%\\&\qquad\qquad\qquad\qquad\qquad\qquad\qquad\qquad\qquad\qquad\qquad\qquad  
+ \| \la x\ra^{-2}(u- u_\rho)\|_{L^2_{t,x}}^2
\end{split}
\]
and \eqref{lsestnt.res} follows.

%%%%%%%%%%%%%%%%%%%%%%%%%%%%%%%%%%%%%%%%%%%%%%%%%%%%%%%%%%%%%%%%%%%%%%%%%%%%%%%%%%%%%%%%5
%%%%%%%%%%%%%%%%%%%%%%%%%%%%%%%%%%%%%%%%%%%%%%%%%%%%%%%%%%%%%%%%%%%%%%%%%%%%%%%%%%%%%%%%%
%%%%%%%%%%%%%%%%%%%%%%%%%%%%%%%%%%%%%%%%%%%%%%%%%%%%%%%%%%%%%%%%%%%%%%%%%%%%%%%%%%%%%%%%%

\bigskip
\newsection{Time independent nontrapping metrics}

The aim of this section is to prove
Theorems~\ref{theorem.PcSmoothing},\ref{theorem.PcSmoothing.res}.
Thus we work with a nontrapping, self-adjoint operator $A$ whose
coefficients are time independent.  We prove
Theorem~\ref{theorem.PcSmoothing} in detail, and then outline the
modifications which are needed for
Theorem~\ref{theorem.PcSmoothing.res}.

\subsection{Proof of Theorem~\ref{theorem.PcSmoothing}}  Here we shall provide the
details for the $n\neq 2$ case.  The general case follows with the obvious logarithmic adjustments
to the $\tX$ spaces in $n=2$.

We break the proof into steps.

%\begin{enumerate}
%%%%%%%%%%%%%%%%%%%%%%%%%%%%%%
  \par{\bf Step 1:} Without any restriction, we assume that $u_0=0$
  and that $u$ is the forward solution to \eqref{main.equation}.
  Nonzero initial data $u_0$ can be easily added in via a $TT^*$
  argument.

%%%%%%%%%%%%%%%%%%%%%%%%%%%%%%
  \par{\bf Step 2:}  We add a damping term to the equation
\[
(D_t+A-i\varepsilon)u_\varepsilon = f
\]
in order to insure global square integrability of the solution
$u_\epsilon$.  Applying our nontrapping estimate \eqref{lsnt} we have
\begin{equation}
\|u_\varepsilon\|_\tX\lesssim
\|f\|_{\tX'}+\|u_\varepsilon\|_{L^2_{t,x}(\R\times B(0,2R))}.
\label{veps}\end{equation}
We want to eliminate the second term on the right (when we add $P_c$
on the left).

%%%%%%%%%%%%%%%%%%%%%%%%%%%%%%
  \par{\bf Step 3:} 
  We want to take a Fourier transform in time and use Plancherel's
  theorem. For this we need to work with Hilbert spaces. These are
  defined using the structure introduced in the previous section.  We
  denote by $\alpha$ a family of positive sequences $(\alpha(k)_j)_{j
    \geq k^-}$ which have sum $1$ for each $k$ and by $\mathcal
  A$ the collection of such sequences. For $\alpha \in \cal A$ we
  define the Hilbert space $\tX_\alpha$ with norm
\[
\| u\|_{\tX_\alpha}^2 = \sum_{k} 2^k\|S_k u\|_{X_{k,\alpha(k)}}^2 + \|\la
x\ra^{-1} u\|_{L^2_{t,x}}^2
\]
as well as its dual $\tX_\alpha'$. 
Since
\[
\| u\|_{\tX} \approx  \sup_{\alpha \in \cal A} \| u\|_{\tX_\alpha},
\qquad
\| u\|_{\tX'} \approx  \inf_{\alpha \in \cal A} \| u\|_{\tX'_\alpha}
\]
we can rewrite \eqref{veps} in the equivalent form
\[
\|u_{\varepsilon}\|_{\tX_\alpha}\lesssim \|f\|_{\tX'_\beta}+
\|u_\varepsilon\|_{L^2_{t,x}(\R\times B(0,2R))}, \qquad \alpha, \beta
\in \cal A.
\]
We denote by $X^0_\alpha$ the spatial version of $X_\alpha$, i.e.
$X_\alpha = L^2_t X_\alpha^0$.  Then we take a time Fourier transform,
and by Plunderer this is equivalent to
\[
\|\hat{u}_\varepsilon\|_{L^2_\tau \tX_\alpha^0}\lesssim 
\|\hat{f}\|_{L^2_\tau (\tX^0_\beta)'} + \|\hat{u}_\varepsilon
\|_{L^2_{\tau,x}(\R\times B(0,2R))}.
\]
This is in turn equivalent to the fixed $\tau$ bound
\[
\|\hat{u}_\varepsilon(\tau)\|_{\tX_\alpha^0}\lesssim
\|\hat{f}(\tau)\|_{(\tX_\beta^0)'}
+\|\hat{u}_\varepsilon(\tau)\|_{L^2(B(0,2R))},
\]
which we rewrite in the form
\[
\|v  \|_{\tX^0_\alpha}\lesssim
\|(A-\tau-i\varepsilon) v \|_{(\tX_\beta^0)'}+\| v \|_{L^2(B(0,2R))},
\]
or, optimizing with respect to $\alpha,\beta \in \cal A$,
\begin{equation}\label{have}
\|v  \|_{\tX^0}\lesssim
\| (A-\tau-i\varepsilon)v \|_{(\tX^0)'}+\| v \|_{L^2(B(0,2R))}.
\end{equation}
A similar
computation shows that the estimate that we want to prove, namely 
\eqref{PcSmoothing} with $u_0=0$, can be rewritten in the equivalent
form 
\begin{equation}\label{want}
  \|P_c v  \|_{\tX^0}\lesssim \|  (A-\tau-i\e) v \|_{(\tX^0)'}
\end{equation}
uniformly with respect to $\tau \in \R$, $\e > 0$.

%%%%%%%%%%%%%%%%%%%%%%%%%%%%%%
   \par{\bf Step 4:}  When $|\tau|$ is large, \eqref{want}
follows from \eqref{have} combined  with the elliptic bound
\begin{equation}\label{ellip}
\tau^{1/4} \|v\|_{L^2(B(0,2R))}\lesssim \|v\|_{\tX^0}+\|(A-\tau-i\e)v\|_{(\tX^0)'}.\end{equation}
To prove this we replace $v$ by $w = (1-\rho) v$ and rewrite 
it in the form
\[
\tau^{1/4} \|w\|_{L^2}\lesssim 
\|w\|_{H^\frac12}+\|(A-\tau-i\e)w\|_{H^{-\frac12}}
\]
for $w$ with compact support. Since
\[
\tau \| w \|_{H^{-\frac32}} \lesssim \|(A-\tau-i\e)w\|_{H^{-\frac32}}
+ \| A w\|_{H^{-\frac32}} \lesssim \|(A-\tau-i\e)w\|_{H^{-\frac12}}
+ \| w\|_{H^{\frac12}},
\]
the bound \eqref{ellip} follows by interpolation.
%%%%%%%%%%%%%%%%%%%%%%%%%%%%%%
\par{\bf Step 5:}
For $\tau$ in a bounded set we argue by contradiction.
If  \eqref{want} does not hold uniformly  then we find sequences 
\[
\varepsilon_n\to 0, \qquad  \tau_n\to \tau, 
\]
and $ v_n \in \tX^0$ with $P_c v_n=v_n$ and
\[
\|(A-\tau_n-i\varepsilon_n)v_n\|_{(\tX^0)'}\to 0,\quad
\|v_n\|_{L^2(B(0,2R))}=1.
\]
On a subsequence we have
\[
v_n\to v\quad \text{weakly* in} \quad \tX^0.
\]
Since $\tX^0 \subset H^{\frac12}_{loc}$, on a subsequence
we have the strong convergence
\[
v_n \to v \qquad \text{in }  L^2_{loc}.
\]
Hence we have produced a function $v $ with 
\begin{equation}
  v \in \tX^0, \qquad P_c v = v, \qquad (A-\tau) v = 0, 
\qquad \| v\|_{L^2(B(0,2R))}=1.
\label{vprop}\end{equation}
Depending on the sign of $\tau$ we consider three cases.

\par{\bf Step 6:}
If $\tau < 0$ then, using the bound \eqref{bc} for the lower order terms
in $A$, we obtain
\[
\| D_i a^{ij} D_j v -\tau v\|_{(\tX^0)'} \lesssim \| v\|_{\tX^0}.
\]
Then
\[
\| v\|_{\tX^0}^2 \gtrsim \la v, D_i a^{ij} D_j v -\tau v \ra 
\gtrsim \|v\|_{H^1}^2 ,
\]
and therefore $v \in L^2$ is an eigenfunction. This contradicts 
the relation $P_c v = v$.

\par{\bf Step 7:} 
If $\tau = 0$ then there is either a zero eigenvalue or a zero
resonance, both of which are excluded by hypothesis.

\par{\bf Step 8:} 
It remains to consider the most difficult case $\tau > 0$. Here 
the properties \eqref{vprop} of $v$ are no longer sufficient 
to obtain a contradiction. Instead we will establish an additional
property of $v$, namely that $v$ satisfies an outgoing radiation
condition. In order to state this, we need an additional 
regularity property for $v$. We define the space $\tX^{0}_{med}$
with norm
\[
\| v \|_{\tX^{0}_{med}} = \| v\|_{L^2(D_{0})} + \|\nabla
v\|_{L^2(D_{ 0})} + \sup_{j > 0} \| |x|^{-\frac12} v\|_{L^2(D_j)} +
  \| |x|^{-\frac12} \nabla v\|_{L^2(D_j)}
\]
which coincides with the $\tX^0$ norm for intermediate frequencies but
improves it at both low and high frequencies.  Then we claim that $v
\in \tX^0_{med}$. More precisely, we will prove the elliptic bound
\begin{equation}
\| v \|_{\tX^{0}_{med}}  \lesssim \|v\|_{\tX^0} +
  \|(A-\tau-i\e)v\|_{(\tX^0)'}, \qquad 0 < \tau_0 < \tau < \tau_1 
\label{vregb} \end{equation}
with implicit constants which may depend on the thresholds
$\tau_0$, $\tau_1$.  

Now we define the closed subspace $\tX^0_{out}$ of $\tX^0$,
\[
\tX^0_{out}=\{v\in \tX^0_{med}\,:\, \lim_{j\to\infty} \|r^{-1/2}(\partial_r
-i\tau^{1/2})v\|_{L^2(D_j)}=0\},
\]
and  also claim that $v$ has the additional property
\begin{equation}
v \in \tX^0_{out}.
\label{vout} \end{equation}
In other words this implies that $v$ is a resonance contained 
inside the continuous spectrum.

We postpone the proof of \eqref{vregb} and \eqref{vout} 
and conclude first our proof by contradiction, by showing
that there are no resonances inside the continuous spectrum.
Such results are known, see for instance \cite{Agmon}, but perhaps
not in the degree of generality we need here. In any case,
for the sake of completeness, we provide a full proof. 

Let $\chi$ be a smooth spherically symmetric increasing bump function
$\chi$ with $\chi(r)\equiv 0$ for $r<1/2$ and $\chi(r) \equiv 1$
for $r>2$. Since $A$ is self-adjoint, for large $j$ we commute
\[
\begin{split}
0=&\ \frac{i}{2}\la [A,\chi(2^{-j}r)]v,v\ra \\
 = &\ \Im\left\la 2^{-j} \chi'(2^{-j} r)
\Bigl(\frac{x_ia^{ij}}{r}\partial_j -i\tau^{1/2}\Bigr)v,v\right\ra
+2^{-j} \tau^{1/2}\la \chi'(2^{-j}r) v,v\ra
\\&\qquad\qquad\qquad\qquad\qquad\qquad\qquad\qquad\qquad\qquad
+2^{-j}\Bigl\la b^i\frac{x_i}{r} \chi'(2^{-j}r) v, v\Bigr\ra .
\end{split}
\]
Using the Schwarz inequality, \eqref{coeffb.M}, and the outgoing radiation condition, we
conclude that
\begin{equation}\label{initial.decay}
\lim_{j\to\infty} \|r^{-1/2} v\|_{L^2(D_j)}=0
\end{equation}
which shows that $v$ has better decay at infinity. We note that 
this is the only use we make of the radiation condition. From this, by
elliptic theory, we also obtain a similar decay for the gradient,
\begin{equation}\label{dinitial.decay}
\lim_{j\to\infty} \|r^{-1/2} \nabla v\|_{L^2(D_j)}=0.
\end{equation}

To conclude we use \eqref{initial.decay} and \eqref{dinitial.decay} to
show that in effect $v \in L^2$; i.e. $v$ is an eigenvalue.  Then by
the results of \cite{KT2} $v$ must be $0$.  Here, we shall again use a
positive commutator argument.  The multiplier we use is the operator
$Q_k$, for some $k\le 0$, in Lemma~\ref{lemma.low.freq} but where for simplicity we set
$\delta = 1$.  We have
\[
0=-2\Im \la Q_k v, (A-\tau)v\ra = \la C_k v,v\ra - 2 
\Im \la Q_k v, (b^j D_j + D_j b^j + c)v\ra
\]
where
\[
 C_k =  i [D_l a^{lm} D_m,Q_k].
\]
The expression of the operator $C_k$ is exactly as in the 
formula \eqref{C} but with unmollified coefficients $a^{ij}$.
The main contribution $C_k^0$ is estimated as there
by
\[
\la C_k^0 v,v\ra \gtrsim \left\la \frac{\alpha(2^k |x|)}{\la 2^k
    x\ra} \nabla v,\nabla v \right\ra, 
\]
while the error terms are bounded by
\[
\left\la \frac{\kappa ( |x|)}{\la 2^k x\ra} \nabla v,\nabla v \right\ra 
\]
respectively 
\[
\left\la \la x \ra^{-2}  v, v \right\ra. 
\]
The expression $\Im \la Q_k v, (b^j D_j + D_j b^j + c)v\ra$ can
also be included in the two error terms.
Thus we obtain
\[
\left\la \frac{\alpha(2^k |x|)}{\la 2^k x\ra} \nabla v,\nabla v \right\ra 
\lesssim \left\la \frac{\kappa ( |x|)}{\la 2^k x\ra} \nabla v,\nabla v \right\ra 
+ \left\la \la x \ra^{-2}  v, v \right\ra. 
\]
For $|x| > 2^{M}$ we have, by \eqref{new.alpha},
\[
\kappa(x) \lesssim \e \alpha(2^k x);
\]
therefore the first term on the right is essentially negligible. We obtain
\[
\int \frac{\alpha(2^k |x|)}{\la 2^k x\ra}  |\nabla v|^2 dx 
\lesssim \int_{D_{ < M}}   |\nabla v|^2 dx 
+ \int \la x \ra^{-2} |v|^2 dx .
\]
At the same time we have
\[
0 = \left\la  \frac{\alpha(2^k |x|)}{\la 2^k x\ra} v, (A-\tau) v  \right\ra,  
\]
which after an integration by parts yields
\[
\tau \int \frac{\alpha(2^k |x|)}{\la 2^k x\ra}  |v|^2 dx 
\lesssim \int \frac{\alpha(2^k |x|)}{\la 2^k x\ra}  |\nabla v|^2 dx 
+ \int  \la x \ra^{-2} |v|^2 dx. 
\]
Combining the two relations we obtain
\[
\int \frac{\alpha(2^k |x|)}{\la 2^k x\ra} ( |\nabla v|^2 +|v|^2 ) dx 
\lesssim \int_{D_{ < M}}   |\nabla v|^2 dx 
+ \int \la x \ra^{-2} |v|^2 dx. 
\]
Finally we let $k \to -\infty$ to obtain
\[
\int  |\nabla v|^2 +|v|^2  dx 
\lesssim \int_{D_{ < M}}   |\nabla v|^2 dx 
+ \int \la x \ra^{-2} |v|^2 dx < \infty 
\]
which shows that $v \in L^2$.  

We note that \eqref{initial.decay} and
\eqref{dinitial.decay} are not used in any quantitative way but serve
only to justify the previous computations. More precisely, one can 
introduce in the computation a cutoff outside a large enough ball
and then pass to the limit. 

It remains to prove \eqref{vregb} and \eqref{vout}.

%%%%%%%%%%%%%%%%%%%%%%%%%%%%%%%%%%%%%%%%%%%%%%%%%%
{\bf Step 9:} Here we prove \eqref{vregb}. We begin with the 
bounds on $v$. This is trivial for the high frequencies of $v$,
\[
 \| S_{>0} v\|_{X^0_0} \lesssim \| v\|_{\tX^0}.
\]
To estimate the low frequencies, we compute
\[
(\tau+i\e)   S_{<0} v = S_{< 0} A v - S_{<0}(A-\tau-i\e)v.
\]
Writing $A$ in the generic form
\[
A =  D^2 a + D b + c, 
\]
we have
\[
\begin{split}
\| S_{<0} v\|_{X^0_{0}} \lesssim &\  \| S_{< 0} D^2 av\|_{X^0_0} +
\| S_{<0} bv \|_{X^0_0} + \| S_{<0} cv\|_{X^0_0} + \|S_{<0} (A-\tau-i\e) v\|_{X^0_0}
\\
 \lesssim &\  \| av\|_{\dX^0} +
\|  bv \|_{X^0_0} + \| cv\|_{X^0_0} + \|(A-\tau-i\e)v\|_{(\tX^0)'}
\\  \lesssim &\  \| v\|_{\dX^0} +
\|  \la x\ra^{-1} v \|_{L^2} + \|(A-\tau-i\e)v\|_{(\tX^0)'}.
\end{split}
\]
Once we control $\| v\|_{X^0_0}$, we can also obtain control of $\| \nabla
v\|_{X^0_0}$ by a straightforward elliptic estimate.

%%%%%%%%%%%%%%%%%%%%%%%%%%%%%%%%%%%%%%%%%%%%%%%%%%
{\bf Step 10:} Here we prove the outgoing radiation condition
\eqref{vout} for $v$. This is obtained from similar outgoing radiation
conditions for the functions $v_n$. However, $v_n$ only converges to
$v$ in a weak sense. Hence we need to produce some uniform estimates
for $v_n$ which will survive in the limit.
\begin{multline}\label{out}
\| r^{-\frac12} (D_r - \tau^\frac12) u\|_{L^2(D_j)}^2 \\\lesssim   
\sum_{k=0}^\infty 
2^{-\delta(k-j)^-} 
\left( \|\la r\ra^\frac12 (A-\tau-i\e) u\|_{L^2(D_k)}  \| \la r\ra^{-\frac12}
  (u,\nabla u)\|_{L^2(D_k)} \right. \\  \left. + 
\kappa_k \| r^{-\frac12} (u,\nabla u)\|_{L^2(D_k)}^2 \right).
\end{multline}
In other words, there is decay when $k<j$. Applying to $v_n$, 
in the weak limit we obtain 
\[
\| r^{-\frac12} (D_r - \tau^\frac12) v \|_{L^2(D_j)}^2 \lesssim   
\sum_{k=0}^\infty 2^{-\delta(k-j)^-} \kappa_k 
\]
which implies \eqref{vout}.

The lower order terms in $A$ can be treated perturbatively 
in \eqref{out}. I.e. they can be included in the right hand side.
Hence without any restriction in generality we assume that
\[
A = D_i a^{ij} D_j.
\]
We use again a positive commutator method.  The multiplier is
the self-adjoint operator
 \[
 Q=b(R)\Bigl(\frac{x_i a^{ij}}{R}D_j - \tau^{\frac{1}{2}}\Bigr) + \Bigl(
 D_j \frac{ a^{ij} x_i}{R} - \tau^{\frac{1}{2}}\Bigr)b(R), \qquad
 R^2=x_i a^{ij}x_j
\]
where the coefficient $b(R)$ is smooth, increasing and satisfies
\[ 
b(R)\approx  \left\{
\begin{array}{lc}
  1 & R>2^{j+2} \cr
  (2^{-j}R)^\delta,    & 1 <  R< 2^{j+2}
\end{array} \right.
\]
with $\delta$ a small parameter.
We write
\begin{equation}\label{commute.eqn}
-2\Im \la Qu, (A-\tau-i\varepsilon)u\ra = \la i [A,Q]u,u\ra - 2\varepsilon \la Qu,u\ra.
\end{equation}
We expect to get the main positive contribution from the first term on
the right. The second term on the right on the other hand is
essentially negative definite due to the fact that its symbol is
negative on the characteristic set of $A - \tau$. Finally, the term on
the left is bounded simply by Cauchy-Schwarz.

To shorten the notations, in the sequel we denote by $E$ error terms
of the form
\[
E = D O( b(R) r^{-1} \kappa(|x|)) D + O(b(R) r^{-1} \kappa(|x|)).
\]
Such terms occur whenever $a^{ij}$ is either differentiated or
replaced by the identity and are easily estimated in terms of the 
right hand side of \eqref{out}.

We evaluate the commutator $ i [A,Q]$.  A similar  computation
was already carried out in \eqref{C.high.calc}, which we reuse with
$k = 0$, $\delta = 1$ and $\phi(r) = b(R)/R$. We obtain
\[
\begin{split}
 i [A,Q]= &\  4 D \frac{b(R)}{R} D + 4 D  x   \left( \frac{b'(R)}{R^2} -
 \frac{b(R)}{R^3}\right) xD - 2\tau^\frac12 \left(\frac{b'(R)}{R} x D + D x \frac{b'(R)}{R}\right)  +E
\\
= &\  2D \left( 2 \frac{b(R)}{R} - b'(R)\right) D - 2 Dx  \left(
    2 \frac{b(R)}{R^3} - \frac{b'(R)}{R^2}\right) xD  
\\&\qquad
+ 
  {b'(R)} (A -\tau) +  (A -\tau) {b'(R)}
%\\ &\ 
+ 2 \left( Dx - \tau^\frac12 r\right ) \frac{b'(R)}{r R}  \left(
   xD -   r \tau^\frac12 \right ) + E.
\end{split}
\]
Our choice of $b$ insures that the coefficient in the first two terms
is positive,
\[
2 \frac{b(R)}{R} - {b'(R)} \geq 0 \qquad R > 1.
\]
Hence we obtain
\[
\la  i [A,Q] u, u\ra \gtrsim 2 \la b'(R) (D_r - \tau^\frac12) u, (D_r -
\tau^\frac12) u\ra + 2 \Re \la (A -\tau-i\epsilon) u,  b'(R) u \ra +
\la Eu,u\ra
\]
where we have inserted a harmless $\epsilon$ term.

It remains to evaluate the second term on the right in
\eqref{commute.eqn}. We have 
\[
\begin{split}
\tau^\frac12 Q = &\ - \Bigl(D_k \frac{x_l a^{kl}}{R}-\tau^{1/2}\Bigr)
b(R) \Bigl(\frac{x_i a^{ij}}{R}D_j -\tau^{1/2}\Bigr)
   + \frac{b(R)}2 (A-\tau) +  (A-\tau) \frac{b(R)}2
\\
&\ - \Bigl(D_i -D_l\frac{a^{lk}x_k x_i}{R^2}\Bigr) a^{ij} b(R)
\Bigl(D_j -\frac{x_j x_m a^{mn}}{R^2}D_n\Bigr) - \frac12 (A b(R)).
\end{split}
\]
The first and third terms are negative while the last term can be
included in $E$. Hence we obtain
\[
\tau^{\frac12} \la Q u,u\ra \leq \Re \la b(R) u, (A-\tau-i\e) u \ra +
\la Eu,u\ra .
\]

Returning to \eqref{commute.eqn}, we insert the bounds 
for the two terms on the right to obtain
\[
\la b'(R) (D_r - \tau^\frac12) u, (D_r -\tau^\frac12) u\ra 
\lesssim  \Re \la (A -\tau-i\epsilon) u,  (2 b'(R)+\e \tau^{-\frac12} b(R)
+i Q)
u \ra + \la Eu,u \ra.
\]
In the region $D_j$, we have $b' \approx 2^{-j} \approx r^{-1}$;
therefore \eqref{out} follows.

\subsection{Proof of Theorem~\ref{theorem.PcSmoothing.res}}

We proceed as in the nonresonant case. The bound \eqref{veps}
is replaced by
\begin{equation}
\|u_\varepsilon\|_\dX\lesssim
\|f\|_{\dX'}+\|u_\varepsilon-u_{\varepsilon \rho}
\|_{L^2_{t,x}(\R\times B(0,2R))}.
\label{veps.res}\end{equation}
Using Plancherel as in Step 3,  this is equivalent to the spatial
bound
\begin{equation}\label{have.res}
\|v  \|_{\dX^0}\lesssim
\| (A-\tau-i\varepsilon)v \|_{(\dX^0)'}+\| v-v_\rho  \|_{L^2(B(0,2R))}
\end{equation}
where $\dX^0$ is the fixed time counterpart of $\dX$.  On the other
hand the estimate that we want to prove, namely
\eqref{PcSmoothing.res} with $u_0=0$, has the
equivalent form
\begin{equation}\label{want.res}
  \|v  \|_{\dX^0}\lesssim \|  (A-\tau-i\e) v \|_{(\dX^0)'}
\end{equation}
uniformly with respect to $\tau \in \R$, $\e > 0$.

For $\tau$ away from $0$ we can easily bound the local average
of $v$. We have
\[
(\tau+i\epsilon) v_\rho = (Av)_\rho - ((A-\tau-i\varepsilon)v )_\rho.
\]
Therefore, by Cauchy-Schwarz,
\[
\tau |v_\rho| \lesssim \| (A-\tau-i\varepsilon)v \|_{(\dX^0)'} + \|
v \|_{L^2(B(0,2R))}.
\]
Hence we are able to bound $v$ in $\tX^0$ as well,
\begin{equation}\label{have.resa}
\|v  \|_{\tX^0}\lesssim
\| (A-\tau-i\varepsilon)v \|_{(\dX^0)'}+\| v  \|_{L^2(B(0,2R))},
\qquad |\tau| > \tau_0.
\end{equation}
Consequently, the argument for large $\tau$ rests unchanged.

Consider now the proof by contradiction. 

In the case $\tau < 0$,
we use the bound \eqref{bnoc} instead of \eqref{bc} for the 
lower order terms and show that $v$ is an eigenvalue.
However, by the maximum principle, there can be no 
negative eigenvalue for $A$.

The case $\tau = 0$ is the interesting one.  Then $v$ satisfies
\[
v \in \dX, \qquad A v = 0, \qquad \| v-v_\rho\|_{L^2(B(0,2R))}=1. 
\]
Hence $v$ is a zero generalized eigenvalue; therefore it must be
constant. But this contradicts the last relation.

Finally, due to \eqref{have.resa}, the case $\tau > 0$ is identical to
the nonresonant case. 

%%%%%%%%%%%%%%%%%%%%%%%%%%%%%%%%%%%%%%%%%%%%%%%%%%
\subsection{ Proof of Remark~\ref{remark.nootherres}}

If $Av = 0$ then from
\[
0 = \la A(v-v_{D_j}),\chi_{<j} (v-v_{D_j}) \ra
\]
and integration by parts, we obtain
\[
\int_{D_{<j}}  |\nabla v|^2 \:dx \lesssim \int_{D_j} |x|^{-2} |v-v_{D_j}|^2 dx.
\]
The right hand side is square summable with respect to $j$;
therefore it decays as $j \to \infty$. We conclude that $\nabla v =
0$, and therefore $v$ is constant.

%%%%%%%%%%%%%%%%%%%%%%%%%%%%%%%%%%%%%%%%%%%%%%%%%%%%%%%%%%%%%%%%%%%%%%%%%%%%%%%%%%%%%%%%%%%%%%%%%%%%%%%%%%%%%%%%%%%%%%%
%%%%%%%%%%%%%%%%%%%%%%%%%%%%%%%%%%%%%%%%%%%%%%%%%%%%%%%%%%%%%%%%%%%%%%%%%%%%%%%%%%%%%%%%%%%%%%%%%%%%%%%%%%%%%%%%%%%%%%%
%%%%%%%%%%%%%%%%%%%%%%%%%%%%%%%%%%%%%%%%%%%%%%%%%%%%%%%%%%%%%%%%%%%%%%%%%%%%%%%%%%%%%%%%%%%%%%%%%%%%%%%%%%%%%%%%%%%%%%%

  \newsection{Strichartz estimates}

  In this section we combine the smoothing estimates of the preceding
  sections with the long-time parametrix construction of \cite{T} to
  obtain the Strichartz estimates of Theorems \ref{main.est.theorem},
  \ref{main.est.theorem.res}, \ref{main.theoremnt},
  \ref{main.theoremnt.res}, \ref{corr.nontrap.Strichartz},
  \ref{corr.nontrap.Strichartz.res}.  We begin by recalling the
  relevant results of \cite{T}. A first result asserts that full
  Strichartz/local smoothing estimates hold under a smallness
  assumptions on $\kappa$ in \eqref{coeff}.

\begin{theorem}[\cite{T}]
  \label{prop.Tataru.parametrix}
  Assume that the coefficients $a^{ij}$ satisfy \eqref{coeff} with
  $\kappa$ sufficiently small and $b=0$, $c=0$.  Then for any
  Strichartz pairs $(p_1,q_1)$, $(p_2,q_2)$, the solution $u$ to
  \eqref{main.equation} satisfies
\begin{equation}
\|u\|_{L^{p_1}_t L^{q_1}_x\cap \dX}\lesssim 
\|u_0\|_{L^2}+
\|f\|_{L^{p_2'}_tL^{q_2'}_x+\dX'}.
\end{equation}
\end
{theorem}
For large $\kappa$, which is the case we are interested in here,
it is shown that

\begin{theorem}[\cite{T}]
  \label{prop.Tataru.parametrix.2}
  Assume that   the coefficients  $a^{ij}$  satisfy \eqref{coeff}  and
  $b=0$, $c=0$.  Then there is a parametrix $K=\sum_{k} K_k S_k$  for $D_t+A$ with 
each $K_k$ localized at frequency $2^k$ so that
the
  following properties hold:
\begin{enumerate}
  \item[(i)] For any Strichartz pairs $(p_1,q_1)$ and $(p_2, q_2)$  we have
    \begin{equation}
      \label{kkf}
\|K_k S_k f\|_{L^{p_1}_tL^{q_1}_x\cap X_k}\lesssim \|S_k f\|_{L^{p_2'}_tL^{q_2'}_x}
    \end{equation}
and
\begin{equation}\label{kf}
\|Kf\|_{L^{p_1}_tL^{q_1}_x\cap \dX}\lesssim \|f\|_{L^{p_2'}_tL^{q_2'}_x}.\end{equation}
\item[(ii)] For any Strichartz pair $(p,q)$, we have
  \begin{equation}
    \label{lperror2b}
   \|((D_t+A)K-I)f\|_{\dX'}\lesssim \|f\|_{L^{p'}_tL^{q'}_x}.
  \end{equation}
\end{enumerate}
\end{theorem}

As a consequence of this, it is also proved in \cite{T}
that
\begin{theorem}[\cite{T}]
  \label{prop.Tataru.xtolp}
  Assume that   the coefficients  $a^{ij}$  satisfy \eqref{coeff}  and
  $b=0$, $c=0$. Then for any Strichartz pair $(p,q)$, we have
\begin{equation}
\| u\|_{L^p_t L^q_x} \lesssim \|u\|_{\dX \cap L^\infty_t L^2_x} + \| Pu\|_{\dX'}. 
\end{equation}
\end{theorem}

These are slight modifications of the results in \cite{T} as our
assumption \eqref{coeff} is not scale invariant and as such we have
modified the definitions of $\tX_k$ and $A_{(k)}$ slightly.  Scale
invariance, however, was only assumed in \cite{T} as a convenience,
and the modifications that are necessary to adapt the proofs to the
current setting are straightforward.

The above results are suitable for the high dimension $n \geq
3$ and for the low dimensional resonant case. However, for the 
low dimensional nonresonant case, we need a modified 
formulation of the last two theorems.

\begin{theorem}%[\cite{T}]
  \label{prop.Tataru.parametrix.nr}
  Assume that   the coefficients  $a^{ij}$  satisfy \eqref{coeff}  and
  $b=0$, $c=0$.  There there is a parametrix $K$  for $D_t+A$ with the
  following properties:
\begin{enumerate}
  \item[(i)] For any Strichartz pairs $(p_1,q_1)$ and $(p_2, q_2)$  we have
\begin{equation}\label{kft}
\|Kf\|_{L^{p_1}_tL^{q_1}_x\cap \tX}\lesssim \|f\|_{L^{p_2'}_tL^{q_2'}_x}.\end{equation}
\item[(ii)] For any Strichartz pair $(p,q)$,
  \begin{equation}
    \label{lperror2bt}
   \|((D_t+A)K-I)f\|_{\tX'}\lesssim \|f\|_{L^{p'}_tL^{q'}_x}.
  \end{equation}
\end{enumerate}
\end{theorem}
As a consequence of this, by the  same duality argument as in \cite{T},
we obtain
\begin{theorem}
  \label{prop.Tataru.xtolp.nr}
  Assume that   the coefficients  $a^{ij}$  satisfy \eqref{coeff}  and
  $b=0$, $c=0$. Then for any Strichartz pair $(p,q)$, we have
\begin{equation}
\| u\|_{L^p_t L^q_x} \lesssim \|u\|_{\tX\cap L^\infty_tL^2_x} + \| Pu\|_{\tX'}. 
\end{equation}
\end{theorem}

\begin{proof}[Proof of Theorem~\ref{prop.Tataru.parametrix.nr}]
  The conclusion of the theorem follows by replacing the parametrix
  $K$ with $(1-T)K+R$, where $T$ and $R$ are linear operators which
  are translation invariant in $t$ and have the following properties:
\begin{equation}
\|(1- T) u\|_{\tX} \lesssim \|u\|_{\dX},
\label{umtu}\end{equation}
\begin{equation}
  \label{lplqtu}
\|(1-T)Kf\|_{L^{p_1}_tL^{q_1}_x}\lesssim \|f\|_{L^{p_2'}_tL^{q_2'}_x},
\end{equation}
\begin{equation}
\|A R f\|_{\dX'}   + \|R f\|_{\tX\cap L^{p_1}_tL^{q_1}_x} \lesssim \|f\|_{L^{p_2'}_tL^{q_2'}_x},
\label{arf}\end{equation}
\begin{equation}
\| AT u\|_{\dX'}+ \| TA u\|_{\dX'}  \lesssim \| u\|_{\dX}, 
\label{atu}\end{equation}
\begin{equation}
\|(T-D_t R)f\|_{\tX'} \lesssim \| f\|_{L^{p'}_tL^{q'}_x}.
\label{tmdtr}\end{equation}

We seek $T$, $R$  of the form
\[
Tu = \sum_{k = -\infty}^0 T_k S_k u, \qquad 
Rf = \sum_{k = -\infty}^0 R_k S_k f
\]
where the operators $T_k$, $R_k$ are localized at frequency $2^k$,
respectively $\geq 2^k$
and are defined by 
\[
T_k  = u(t,0) \phi_k, \qquad R_k f = \phi_0(x)D_t^{-1} S^t_{>0}
f(t,0)  -\sum_{j=k}^{-1} (\phi_{j+1}(x)-\phi_j(x))D_t^{-1} S^t_{>2j} f(t,0) 
\]
with $\phi_k(x) = \phi(2^{k} x)$ and
\[
\phi(0) = 1, \qquad \supp \hat{\phi} \subset \{|\xi| \in [1/2,2]\}.
\]

Notice that $Tu=u^{in}$ with $u^{in}$ as in Section \ref{embxs}.  As such, the bound
\eqref{umtu} follows directly from \eqref{Tkbdd} and \eqref{l2uout}.
The bound \eqref{lplqtu} follows similarly using a Bernstein bound, Littlewood-Paley theory, and
\eqref{kkf}.
For \eqref{atu} we use Proposition~\ref{lemma.A.to.Ak} to replace 
$A$ by $\sum A_{(k)} S_k$. Then we use the spatial localization
coming from $T$, \eqref{bernstein}, and the two derivatives gain from $ A_{(k)}$.

We consider now the $\dX$ bounds in \eqref{arf}.  For the second term in the left of \eqref{arf}, 
using  Bernstein's inequality twice  yields
\begin{align*}
\left\| (\phi_{j+1}(x)-\phi_j(x))D_t^{-1}S_{>2j}^t
(S_k f)(t,0)\right\|_{X_j}
&\lesssim 2^{\frac{2-n}2 j} 2^{2j(-1+\frac{1}{p_2'}-\frac{1}{2})}
\|S_k f(t,0)\|_{L^{p_2'}_t} \\
&\lesssim   2^{\frac{2-n}2j} 2^{2j(-1+\frac{1}{p_2'}-\frac{1}{2})}
2^{\frac{n}{q_2'}k} \|S_k f\|_{L^{p_2'}_tL^{q_2'}_x} \\
&= 
2^{\frac{n}{q_2'}(k-j)} \|S_k f\|_{L^{p_2'}_tL^{q_2'}_x}.
\end{align*}
The $j=0$ term in $R_k$ is estimated in a similar fashion.  Summing
with respect to $k \leq j \leq 0$ we use the off-diagonal decay to
obtain
\begin{align*}
\|Rf\|_X &\lesssim   \left(\sum_{j=-\infty}^0 \left(\sum_{k=-\infty}^j 
2^{\frac{n}{q_2'}(k-j)} \|S_k f\|_{L^{p_2'}_tL^{q_2'}_x}\right)^2\right)^\frac12\\
&\lesssim \left(\sum_{k=-\infty}^0 \|S_k f\|^2_{L^{p_2'}_tL^{q_2'}_x}\right)^\frac12.
\end{align*}
 The bound $X$ bound for the second term in the left of \eqref{arf} then follows from Littlewood-Paley theory.
The $L^{p_1}_tL^{q_1}_x$ estimate follows from similar applications of Bernstein estimates and
Littlewood-Paley theory.
  
For the first term in the left of \eqref{arf}, we may apply Proposition \ref{lemma.A.to.Ak} to again
replace $A$ by $\sum A_{(k)}S_k$.  As the derivatives in $A_{(k)}$ yield a $2^{2k}$ factor, the estimate
for the first term in \eqref{arf}  follows from a very similar argument.

In order to complete the proof of \eqref{arf}, we examine the $L^2$
part of the $\tX$ norm.  We may first apply \eqref{Hardy} and
\eqref{txdx} to reduce the problem to the bound
\[
\|\sum_{k<0} R_k S_k f\|_{L^2_{t,x}(\{|x|\le 1\})}\lesssim
\|f\|_{L^{p_2'}_tL^{q_2'}_x}
\]
in dimensions $n=1,2$.
Here we use the fact that
$\phi_{j+1}(0)-\phi_j(0)=0$.  
Using this gain in a fashion similar to that from Section \ref{embxs},
we have 
\[
\|\phi_{j+1}-\phi_j\|_{L^2(\{|x|\le 1\})}\lesssim 2^j.
\]
Thus, arguing as above,
\begin{align*}
\|R_k S_k f\|_{L^2(\{|x|\le 1\})}&\lesssim \sum_{j\ge k} 2^j 2^{2j(-1+\frac{1}{p_2'}-\frac{1}{2})}
2^{\frac{n}{q_2'}k}\|S_k f\|_{L^{p_2'}_tL^{q_2'}_x}\\
&\lesssim 2^{\frac{n}{2}k}\|S_k f\|_{L^{p_2'}_t L^{q_2'}_x}.
\end{align*}
This can clearly be summed to yield the desired bound.

It remains to prove \eqref{tmdtr}. For this we will show the bound
\begin{equation}
\|\la x \ra (T-D_t R)f\|_{L^2} \lesssim \| f\|_{L^{p'}_tL^{q'}_x}.
\label{fdsa}\end{equation}
We have 
\[
(T-D_t R)f = - \sum_{k < 0} \left(\phi_0S^t_{\le 0} (S_k f)(t,0) 
+\sum_{j= k}^{-1} (\phi_{j+1}-\phi_j)S^t_{\le 2j} (S_k f)(t,0)\right).
\]
Arguing as above we obtain
\[
 \| (\phi_{j+1}-\phi_j)S^t_{\le 2j} (S_k f)(t,0)\|_{L^2} 
\lesssim 2^j 2^{\frac{n}{q_2'}(k-j)} \|S_k f\|_{L^{p_2'}_tL^{q_2'}_x}
\]
respectively 
\[
\| x (\phi_{j+1}-\phi_j)S^t_{\le 2j} (S_k f)(t,0)\|_{L^2}
\lesssim 2^{\frac{n}{q_2'}(k-j)} \|S_k f\|_{L^{p_2'}_tL^{q_2'}_x}
\]
and similarly for the $j=0$ term. Then \eqref{fdsa} is obtained by
summation using the off-diagonal decay and Littlewood-Paley theory.
\end{proof}

Theorems~\ref{prop.Tataru.parametrix.nr}, \ref{prop.Tataru.xtolp.nr}
will allow us to derive Theorems \ref{main.est.theorem},
\ref{main.theoremnt}, \ref{corr.nontrap.Strichartz} from
Theorems~\ref{main.ls.theorem}, \ref{main.ls.theoremnt},
\ref{theorem.PcSmoothing}.  Similarly,
Theorems~\ref{prop.Tataru.parametrix}, \ref{prop.Tataru.xtolp}
will allow us to derive Theorems \ref{main.est.theorem.res},
\ref{main.theoremnt.res}, \ref{corr.nontrap.Strichartz.res} from
Theorems~\ref{main.ls.theorem.res}, \ref{main.ls.theoremnt.res},
\ref{theorem.PcSmoothing.res}.

\subsection{Proof of Theorems~\ref{main.theoremnt},
  \ref{corr.nontrap.Strichartz},~\ref{main.theoremnt.res},
  \ref{corr.nontrap.Strichartz.res}}

The four proofs are almost identical, so we discuss only the first
theorem.  Suppose the function $u$ solves
\[
Pu = f +g, \qquad f \in \tX', \quad g \in L^{p'_2}_t L^{q'_2}_x
\]
with initial data 
\[
u(0) = u_0.
\]
We let $K$ be the parametrix of
Theorem~\ref{prop.Tataru.parametrix.nr} and denote
\[
v = u - K g.
\]
Then
\[
Pv = f + g -PKg, \qquad v(0) = u(0) - Kg(0).
\]
Using the bounds \eqref{bc}, \eqref{kft}, and \eqref{lperror2bt}, we obtain
\[
\| v(0)\|_{L^2} + \| Pv\|_{\tX'} \lesssim \|u(0)\|_{L^2} + 
\| f\|_{\tX'} + \|g\|_{L^{p'_2}_t L^{q'_2}_x}.
\]
Then  Theorem~\ref{main.ls.theoremnt} gives
\[
\| v\|_{L^\infty_t L^2_x \cap \tX}  + \| Pv\|_{\tX'} \lesssim \|u(0)\|_{L^2} + 
\| f\|_{\tX'} + \|g\|_{L^{p'_2}_t L^{q'_2}_x} + \|v\|_{L^2_{t,x}(A_{<2R})}.
\]
Hence by \eqref{bc} and Theorem~\ref{prop.Tataru.xtolp.nr} it follows that
\[
\| v\|_{L^\infty_t L^2_x \cap \tX}+ \|v\|_{L^{p_1}_t L^{p_2}_x}
 \lesssim \|u(0)\|_{L^2} + 
\| f\|_{\tX'} + \|g\|_{L^{p'_2}_t L^{q'_2}_x} + \|v\|_{L^2_{t,x}(A_{<2R})}.
\]
Using again \eqref{lperror2bt} we return to $u$ to obtain
\[
\| u\|_{L^\infty_t L^2_x \cap \tX}+ \|u\|_{L^{p_1}_t L^{p_2}_x}
 \lesssim \|u(0)\|_{L^2} + 
\| f\|_{\tX'} + \|g\|_{L^{p'_2}_t L^{q'_2}_x} + \|u\|_{L^2_{t,x}(A_{<2R})}
\]
concluding the proof of the Theorem.

\subsection{Proof of Theorem~\ref{main.est.theorem}}
 Suppose the function $u$ solves
\[
Pu = f +\rho g, \qquad f \in \tX'_e, \quad g \in L^{p'_2}_t L^{q'_2}_x
\]
with initial data 
\[
u(0) = u_0.
\]
We consider two
additional spherically symmetric cutoff functions $\rho_1$ and
$\rho_2$ supported in $\{|x| > 2^M\}$ so that $\rho_2 =1$ in the
support of $\rho_1$ and $\rho_1 =1$ in the
support of $\rho$.

 Let $K$ be the parametrix of
Theorem~\ref{prop.Tataru.parametrix.nr} and denote
\[
v = u - \rho_1 K \rho g.
\]
Then
\[
Pv = f +\rho_2(\rho_1( \rho g - P K \rho g)  - [P,\rho_1]K\rho g)
, \qquad v(0) = u(0) - \rho_1 K \rho g(0).
\]
Using the bounds \eqref{bc}, \eqref{kft}, and \eqref{lperror2bt}, we obtain
\[
\| v(0)\|_{L^2} + \| Pv\|_{\tX'_{e2}} \lesssim \|u(0)\|_{L^2} + 
\| f\|_{\tX'_e} + \|g\|_{L^{p'_2}_t L^{q'_2}_x}
\]
where $\tX'_{e2}$ is similar to $\tX'_{e}$ but with $\rho$ replaced
by $\rho_2$. Then we can apply Theorem~\ref{main.ls.theorem}
to $v$ to obtain
\[
\| v\|_{L^\infty_t L^2_x \cap \tX_{e}}
 + \| Pv\|_{\tX'_{e2}} \lesssim \|u(0)\|_{L^2} + 
\| f\|_{\tX'_e} + \|g\|_{L^{p'_2}_t L^{q'_2}_x}+ \|v\|_{L^2_{t,x}(|x|\le 2^{M+1})}.
\]
We truncate $v$ with $\rho$ and compute
\[
P \rho v = [P,\rho]v + \rho P v .
\]
Then we can estimate
\[
\| v\|_{L^\infty_t L^2_x} + \|\rho v\|_{\tX}
 + \| P(\rho v)\|_{\tX'} \lesssim \|u(0)\|_{L^2} + 
\| f\|_{\tX'_e} + \|g\|_{L^{p'_2}_t L^{q'_2}_x} + \|v\|_{L^2_{t,x}(|x|\le 2^{M+1})}.
\]
Hence by \eqref{bc} and Theorem~\ref{prop.Tataru.xtolp.nr} applied to $\rho v$,
we obtain
\[
\|v\|_{L^\infty_t L^2_x} + \|\rho v\|_{\tX \cap L^{p_1}_t L^{q_1}_x}
 \lesssim \|u(0)\|_{L^2} + 
\| f\|_{\tX'_e} + \|g\|_{L^{p'_2}_t L^{q'_2}_x}+ \|v\|_{L^2_{t,x}(|x|\le 2^{M+1})}.
\]
Finally, we use \eqref{kft} to return to $u$ and obtain
\[
\|u\|_{L^\infty_t L^2_x} + \|\rho u\|_{\tX \cap L^{p_1}_t L^{q_1}_x}
\lesssim \|u(0)\|_{L^2} + 
\| f\|_{\tX'_e} + \|g\|_{L^{p'_2}_t L^{q'_2}_x}+ \|u\|_{L^2_{t,x}(|x|\le 2^{M+1})},
\]
concluding the proof of the Theorem.

\subsection{Proof of Theorem~\ref{main.est.theorem.res}}

The argument is similar to the one above. The chief difference 
is that we can no longer use the truncations by $\rho$, $\rho_1$,
$\rho_2$ and instead we use the modified truncation operators
such as $T_\rho$.

 Suppose the function $u$ solves
\[
Pu = f +\rho g, \qquad f \in \dX'_e, \quad g \in L^{p'_2}_t L^{q'_2}_x
\]
with initial data 
\[
u(0) = u_0.
\]
We let $K$ be the parametrix of
Theorem~\ref{prop.Tataru.parametrix} and denote
\[
v = u - T_{\rho_1} K \rho g
\]
Then we can write
\[
Pv = f +T_{\rho_2}(T_{\rho_1}( \rho g - P K \rho g)  - 
[P,T_{\rho_1}]K\rho g)
, \qquad v(0) = u(0) - T_{\rho_1}K\rho g(0).
\]
Here we compute the commutator
\[
[A,T_{\rho_1}] w = A \rho_1 (w-w_{\rho_1})  - \rho_1 A (w-w_{\rho_1}) -
(1-\rho)(Aw)_{\rho_1}
= [A,\rho_1] (w -w_{\rho_1}) - (1-\rho)(Aw)_{\rho_1}.
\]
Also we have 
\[
(Aw)_{\rho_1} = c_\rho \int (1-\rho_1)A(w-w_{\rho_1}) dx
 = -c_\rho \int (w-w_{\rho_1}) A\rho_1 dx.
\]
Then using the bounds \eqref{bnoc}, \eqref{kf}, and \eqref{lperror2b}, we obtain
\[
\| v(0)\|_{L^2} + \| Pv\|_{\dX'_{e2}} \lesssim \|u(0)\|_{L^2} + 
\| f\|_{\dX'_e} + \|g\|_{L^{p'_2}_t L^{q'_2}_x}.
\]
By Theorem~\ref{main.ls.theorem} for $v$ we get
\[
\| v\|_{L^\infty_t L^2_x \cap \dX_{e}}
 + \| Pv\|_{\dX'_{e2}} \lesssim \|u(0)\|_{L^2} + 
\| f\|_{\dX'_e} + \|g\|_{L^{p'_2}_t L^{q'_2}_x}+ \|(1-\rho)(v-v_\rho)\|_{L^2_{t,x}}.
\]
We truncate $v$ with $T_\rho$ and compute as above the 
commutator $[P,T_\rho]$.
Then we estimate
\[
\| v\|_{L^\infty_t L^2_x} + \|T_\rho v\|_{\dX}
 + \| P(T_\rho v)\|_{\dX'} \lesssim \|u(0)\|_{L^2} + 
\| f\|_{\dX'_e} + \|g\|_{L^{p'_2}_t L^{q'_2}_x} + \|(1-\rho)(v-v_\rho)\|_{L^2_{t,x}}.
\]
Hence by \eqref{bnoc} and Theorem~\ref{prop.Tataru.xtolp} applied to $T_\rho v$,
we obtain
\[
\|v\|_{L^\infty_t L^2_x} + \|T_\rho v\|_{\dX \cap L^{p_1}_t L^{q_1}_x}
 \lesssim \|u(0)\|_{L^2} + 
\| f\|_{\dX'_e} + \|g\|_{L^{p'_2}_t L^{q'_2}_x}+ \|(1-\rho)(v-v_\rho)\|_{L^2_{t,x}}.
\]
Finally, we use \eqref{kf} to return to $u$ and obtain
\[
\|u\|_{L^\infty_t L^2_x} + \|T_\rho u\|_{\dX \cap L^{p_1}_t L^{q_1}_x}
\lesssim \|u(0)\|_{L^2} + 
\| f\|_{\dX'_e} + \|g\|_{L^{p'_2}_t L^{q'_2}_x}+ \|(1-\rho)(u-u_\rho)\|_{L^2_{t,x}},
\]
concluding the proof of the Theorem.

\bigskip
%%%%%%%%%%%%%%%%%%%%%%%%%%%%%%%%%%%%%%%%%%%%%%%%%%%%%%%%%%%%%%%%%%%%%%%%%%%%%%%%%%%%%%%%%%%%%%%%
%%%%%%%%%%%%%%%%%%%%%%%%%%%%%%%%%%%%%%%%%%%%%%%%%%%%%%%%%%%%%%%%%%%%%%%%%%%%%%%%%%%%%%%%%%%%%%%%
%%%%%%%%%%%%%%%%%%%%%%%%%%%%%%%%%%%%%%%%%%%%%%%%%%%%%%%%%%%%%%%%%%%%%%%%%%%%%%%%%%%%%%%%%%%%%%%%

\end{document}